%% file: main.tex
\documentclass{amsart}
\usepackage[utf8]{inputenc}
\usepackage{amstext}
\usepackage{amsthm}
\usepackage{amssymb}
\usepackage{stmaryrd}

\makeatletter
\numberwithin{equation}{section}
\numberwithin{figure}{section}

\input{macroAdd.tex}

\input{macroHybrid.tex}

\thanks{}

\makeatother

\begin{document}
\title[Cluster canonical bases]{An introduction to representation-theoretic canonical bases of cluster
algebras}
\author{Fan Qin}
\email{qin.fan.math@gmail.com}
\begin{abstract}
This is an introduction to cluster algebras and their common triangular
bases. These bases are Kazhdan-Lusztig-type and serve as the canonical
bases of cluster algebras from the representation-theoretic point
of view. We review seeds associated with signed words, cluster operations
(freezing and base change), the extension technique from finite to
infinite rank, which are convenient tools to study these algebras
and their bases. We present recent results on the topic and provide
examples. An appendix collects detailed examples of some cluster algebras
from Lie theory, including those from Lie groups and representations
of quantum affine algebras.
\end{abstract}

\maketitle
\tableofcontents{}

\section{Introduction}

\label{sec:intro} 

\subsection{An informal introduction to cluster algebras}

Let us begin with an informal but intuitive introduction to cluster
algebras, originally invented by Fomin and Zelevinsky \cite{fomin2002cluster}.
Recall that an $m$-dimensional smooth real manifold takes the form
$\mathcal{M}=\cup_{\sd}V_{\sd}$, such that $V_{\sd}$ are open subsets,
$\sd$ denote their indices, and for each $V_{\sd}$, we have a homeomorphism:
\[
(x_{1},\ldots,x_{m}):V_{\sd}\simeq\R^{m}.
\]
The pair $(V_{\sd},(x_{1},\ldots,x_{m}))$ is called a chart. We might
call $(x_{1},\ldots,x_{m})$ its cluster of local coordinates. Furthermore,
we require that the transition maps (change of local coordinates)
are smooth. The collection of all charts is called the atlas of $\cM$.

In mathematics and physics (for example, in algebraic geometry and
in quantum mechanics), instead of working directly with a geometric
object, it is often useful to consider the algebra of functions on
it. For a smooth manifold $\cM$, we may consider the algebra $C^{\infty}(\cM)$
of smooth functions instead. This viewpoint also leads to a quantization:
a commutative algebra $C^{\infty}(\cM)$ endowed with a Poisson bracket
could be replaced by a non-commutative one.

Cluster algebras can be understood in a similar spirit. A rank-$m$
cluster algebra $\alg$ is an algebra endowed with clusters of distinguished
elements called cluster variables: 
\begin{enumerate}
\item For any chosen cluster $(x_{1},\ldots,x_{m})$, every element $f\in\alg$
could be expressed as a rational function in $x_{1},\ldots,x_{m}$. 
\item We impose additional requirements:
\begin{enumerate}
\item[(2a)]  There is an initial cluster, and all other clusters are obtained
recursively from it. 
\item[(2b)]  The transition maps between adjacent clusters (in the recursive
process) have a specific form (see \eqref{eq:exchange-relation}).
\end{enumerate}
\end{enumerate}
Here we work with rational functions instead of smooth functions because
we are in an algebraic-geometric setting. The recursive construction
(2a) was an essential part of the original definition by Fomin and
Zelevinsky \cite{FominZelevinsky03}. It is natural to ask whether
these algebras can be generalized by allowing more general forms of
transition maps. This leads to the notion of generalized cluster algebras;
see \cite{chekhov2014teichmuller}\cite{gekhtman2020generalized}.

In brief, a cluster algebra is an algebra endowed with an atlas (a
collection of clusters) satisfying certain additional requirements.
From this point of view, cluster algebras are very fundamental objects.
And unsurprisingly they appear in many, sometimes unexpected, areas
of mathematics. Their atlas might come from coordinate systems, wall-chamber
structures, $t$-structures of triangulated categories, and so on.

We say that an algebra has a cluster structure if it can be endowed
with an atlas making it into a cluster algebra. Not all algebras admit
cluster structures. And while all cluster algebras have rich structures,
some possess stronger properties. Many fascinating results have been
obtained for cluster algebras arising from quiver representations,
Lie theory, and (higher) Teichmüller theory; see \cite{fomin2010total}
\cite{Keller08Note} \cite{keller2011cluster} \cite{reiten2010cluster}
\cite{leclerc2010cluster} \cite{thurston2014positive} \cite{muller2016skein}~\cite{fomin2016introduction,fomin2017introduction,fomin2020introduction,fomin2021introduction}
for an incomplete list of surveys and accessible papers on these topics.

Cluster algebras are interesting objects for at least the following
reasons: 
\begin{itemize}
\item They possess rich and beautiful structures. 
\item They provide a common platform for connecting and comparing theories
form different areas. 
\item They allow the discovery of new phenomenon even in well-studied mathematical
objects. 
\end{itemize}
While cluster algebras were not invented for having applications in
other fields, they have proven useful in solving open problems. For
example, they have been used to verify the periodicity conjecture
of $T$-systems (certain discrete dynamical systems)\cite{kuniba2011t}\cite{inoue2010periodicities}\cite{keller2013periodicity},
to show the existence of infinite Lagrangian fillings \cite{casals2022infinitely},
and to prove conjectural properties of the $(q,t)$-characters of
simple modules of quantum affine algebras \cite{fujita2022isomorphisms}.

\subsection{Canonical bases of cluster algebras}

Cluster algebras were invented by Fomin and Zelevinsky \cite{fomin2002cluster}
as a framework to study the total positivity (see \cite{Lusztig96})
and the dual canonical bases of quantized enveloping algebras (see
\cite{Lusztig90,Lusztig91}\cite{Kas:crystal}). These algebras admit
natural quantizations \cite{BerensteinZelevinsky05}. They possess
distinguished elements called cluster monomials, namely monomials
in $x_{1},\ldots,x_{m}$ from the same cluster, which can be computed
recursively.

Fomin and Zelevinsky have the following ambitious expectations \cite{fomin2002cluster}:
\begin{itemize}
\item For many varieties $\AVar$ arising from Lie theory, the coordinate
ring $\kk[\AVar]$ is a cluster algebra (we might take $\kk=\C$ in
the classical case and $\kk=\Q(v)$ or $\Q[v^{\pm}]$ in the quantum
case).
\item The cluster algebra has a basis that contains all cluster monomials
and serves as an analog of the dual canonical basis $\dCan$ of the
quantized enveloping algebra $\kk[N]$.
\end{itemize}
The progress on the second expectation was slow for a long time. The
cluster structure on the quantum unipotent subgroup $\kk[N(w)]$ was
developed in \cite{GeissLeclercSchroeer11}\cite{GY13,goodearl2020integral}
(recall that $N=N(w_{0})$ when the associated Lie algebra $\frg$
is semisimple). For the quantized coordinate ring $\kk[N(w)]$ of
a unipotent subgroup $N(w)$ associated with any Kac-Moody algebra
$\frg$, the original dual canonical basis $\dCan$ is known to meet
this expectation \cite{qin2020dual} (see also \cite{qin2017triangular}\cite{Kang2018}\cite{mcnamara2021cluster}\footnote{\cite{qin2017triangular} proved the $ADE$ case and the case where
$w$ has an $\uc$-adapted word and $\frg$ is symmetric Kac-Moody
(Section \ref{subsec:Cluster-algebras-from-q-aff-alg}), and \cite{qin2020dual}
treated all cases. Based on a completely different approach, \cite{Kang2018}
proved the symmetric Kac--Moody case, and \cite{mcnamara2021cluster}
generalized their results to all cases.}). 

A key ingredient in the approach of \cite{qin2017triangular,qin2020dual}
is the definition of the common triangular basis $\can$ for a large
family of quantum cluster algebras. This basis is a Kazhdan-Lusztig-type
basis for cluster algebras (see Remark \ref{rem:BZ-triangular-basis}
for an earlier construction by Berenstein and Zelevinsky). \cite{qin2020dual}
proved that the dual canonical basis $\dCan$ of any $\kk[N(w)]$
is the common triangular basis of $\kk[N(w)]$ (Theorem \ref{thm:tri-basis-is-canonical}).
This result suggests that the common triangular basis $\can$ provides
a natural analogue of the dual canonical basis $\dCan$ in the framework
of cluster theory. Recently, \cite{qin2023analogs} showed that $\can$
indeed meets the second expectation for almost all known cluster algebras
from Lie theory (Theorem \ref{thm:Lie-tri-basis}). Moreover, when
$\frg$ is a symmetric Kac-Moody algebra, these bases admit (quasi-)categorifications,
i.e., they correspond to simple objects in an appropriate monoidal
category (i.e, tensor category). See Theorem \ref{thm:quasi-categorification-Lie},
where we use a weaker version of the categorification in \cite{HernandezLeclerc09}
(Section \ref{subsec:Quasi-categorifications}).

In representation theory, we know several families of canonical bases,
including the canonical bases $\mathbf{C}$ for Hecke algebras, the
canonical bases $\bCan$ for quantized enveloping algebras $U_{q}(\mathfrak{n})$,
and their dual bases $\dCan$ (dual canonical bases) for the quantized
coordinate rings $\kk[N]$. These bases are of Kazhdan-Lusztig type
(Section \ref{sec:A-review-Kazhdan-Lusztig}), and they often correspond
to \emph{minimal }objects in an appropriate module tensor category
(i.e., the simple modules or the indecomposable projective modules).
From the above discussion, we see that the common triangular basis
$\can$ of a cluster algebra shares similar properties. Therefore,
we might call $\can$ the \emph{representation-theoretic canonical
basis }of the cluster algebra.

\begin{rem}[Different canonical bases]

The word ``canonical'' in ``canonical basis'' does not seem to
have a precise definition. We might interpret it as ``the most natural'',
``the best'', or, from a more aesthetic perspective, ``the most
beautiful''. 

On the one hand, from the representation-theoretic point of view,
the dual canonical basis $\dCan$ of $\kk[N(w)]$ (equivalently, the
common triangular basis $\can$ of $\kk[N(w)]$) is its canonical
basis, at least when $\frg$ is of symmetric type.\footnote{For non-symmetric $\frg$, the simple modules of quiver Hecke algebras
seem to provide a better basis \cite{KhovanovLauda08,khovanov2011diagrammatic}\cite{Rouquier08}.
Though our knowledge of this basis is still limited and its relation
with cluster theory is unknown.}

On the other hand, there is yet another canonical basis for cluster
algebras: the theta basis $\vartheta$ proposed by \cite{gross2018canonical}\footnote{For cluster algebras arsing from Teichmüller theory, the canonical
basis introduced in \cite{fock2006moduli} (or the bracelets basis
in \cite{musiker2013bases}) is the theta basis $\vartheta$ in \cite{gross2018canonical};
see \cite{MandelQin2021}. Also, see \cite{davison2019strong} for
the quantum theta basis.}. This basis, constructed from geometric considerations, has exceptionally
strong properties, and we might call it the \emph{geometric canonical
basis}. 

Examples suggest that the following phenomena might hold in general:
\begin{itemize}
\item The two canonical bases differ when the cluster algebra has infinitely
many cluster variables.
\item Each element of $\can$ is a nonnegative linear combination of the
basis elements of $\vartheta$ (see \cite[Section 3]{qin2021cluster}),
when the initial seed of the cluster algebra is associated with a
skew-symmetric matrix (Section \ref{sec:Cluster-algebras}).
\end{itemize}
At present, there is insufficient evidence to determine whether one
basis could be regarded as more canonical than the other.

In addition, there is a mysterious phenomenon. On the one hand, the
elements of $\can$ correspond to \emph{minimal }objects in a module
category---either as established theorems or as guiding principles.
On the other hand, they admit nonnegative decompositions into the
theta basis elements, which are the \emph{atoms }for a cluster algebra
(see the atomic basis in \cite{lee2014greedy}\cite{lee2014greedyPNAS}\cite{cheung2015greedy}\cite{mandel2017theta}).
It is desirable to understand and interpret the decomposition.

\end{rem}

\subsection{Contents.}

In Section \ref{sec:Cluster-algebras}, we give basic definitions
and properties of cluster algebras.

In Section \ref{sec:Seeds-associated-with-words}, we introduce seeds
associated with signed words. They often appear in cluster algebras
from Lie theory. We also provide examples.

In Section \ref{sec:Cluster-operations}, we introduce operations
on cluster algebras (freezing and base change). This part provides
the key tools for studying bases of cluster algebras from Lie theory
in \cite{qin2023analogs}. It is technical, and the reader could skip
it without without loss of continuity in the subsequent sections.

In Section \ref{sec:A-review-Kazhdan-Lusztig}, we review basics of
Kazhdan-Lusztig-type bases and their categorifications.

In Section \ref{sec:Common-triangular-bases}, we introduce the common
triangular bases $\can$ for cluster algebras and present known results:
their existence (Theorems \ref{thm:tri-basis-existence-criterion}
\ref{thm:tri-basis-is-canonical} \ref{thm:Lie-tri-basis}), quasi-categorifications
(Theorem \ref{thm:quasi-categorification-Lie}), and in the special
cases, standard bases (Theorem \ref{thm:std-bases}) and categorifications
(Theorem \ref{thm:categorify-dBS}).

In Section \ref{sec:Extension-to-infinite}, we present the technique
of extending finite rank cluster algebras to infinite rank ones. In
Example \ref{eg:uniqe-A1-quantization}, we apply it to obtain the
quantization matrices of an infinite rank cluster algebra.

In Appendix \ref{sec:Examples-of-Cluster}, we review some cluster
algebras arising from Lie theory: those from Lie group $G$, double
Bruhat cells $G^{u,w}$, and representations of quantum affine algebras.
We also provide examples in details.

\subsection{Convention}\label{subsec:Convention}

All vectors will be column vectors unless otherwise specified.

Let $\sigma$ be a permutation of an index set $I$. For any $m=(m_{i})_{i\in I}\in\R^{I}$,
we define $\sigma m\in\R^{I}$ such that $(\sigma m)_{\sigma i}:=m_{i}$.
Similarly, let $B=(b_{ij})_{i,j\in I}$ be a matrix. We define $\sigma B=(b'_{ij})_{i,j\in I}$
such that $b'_{\sigma i,\sigma j}:=b_{ij}$. We use $\col_{k}B$ to
denote the $k$-th column vector of $B$. For $I_{1},I_{2}\subset I$,
we denote the submatrix $B_{I_{1},I_{2}}:=(b_{ij})_{i\in I_{1},j\in I_{2}}$.

We use $\Id_{n}$ to denote a rank $n$ identity matrix, and $\Id_{I}$
to denote the identity matrix on the index set $I$.

For $I'\subset I$, we use $\pr_{I'}$ to denote the natural projection
from $\R^{I}$ to $\R^{I'}$.

For any vector $m=(m_{i})_{i\in I}\in\R^{I}$, its support is defined
to be $\supp m:=\{i\in I|m_{i}\neq0\}$.

\section{Cluster algebras}\label{sec:Cluster-algebras}

\subsection{Definition of cluster algebras}

\subsubsection*{Seeds}

Let us define cluster algebras following the convention in \cite{qin2024infinite}.
We choose and fix a countable set $I$,\footnote{See \cite{gratz2015cluster} for general infinite rank cluster algebras
and \cite[Remark 2.2 Remark 3.18]{gratz2015cluster} for a discussion
on the countability.} called the set of vertices. We further choose a partition $I=I_{\ufv}\sqcup I_{\fv}$,
where $I_{\ufv}$ denotes the its unfrozen part and $I_{\fv}$ the
frozen part. Choose and fix strictly positive integers $\sym_{i}$,
$i\in I$.

Assume $(b_{ij})_{i,j\in I}$ is an $I\times I$ $\frac{\Z}{2}$-matrix
such that the following conditions hold:
\begin{enumerate}
\item $\sym_{i}b_{ij}=-\sym_{j}b_{ji}$ for any $i,j$;
\item $b_{ik}\in\Z$ if $k\in I_{\ufv}$;
\item It is locally finite, i.e., for any $i$, only finitely many $b_{ij}$
are nonzero and only finitely many $b_{ji}$ are nonzero.
\end{enumerate}
We will denote $\tB:=(b_{ik})_{i\in I,k\in I_{\ufv}}$. $B:=\tB_{I_{\ufv}\times I_{\ufv}}$
is called the principal part of $\tB$.

A (classical) seed $\sd$ is a collection $(I,I_{\ufv},(\sym_{i})_{i\in I},\tB,(x_{i})_{i\in I})$,
where $\tB$ is called its $B$-matrix, and $x_{i}$ are indeterminates
called its cluster variables. $(x_{i})_{i\in I}$ is called the cluster.
We will call $x_{j}$, $j\in I_{\fv}$ its frozen variables. The rank
of $\sd$ is defined to be the cardinality $|I|$.

\begin{rem}

For our purpose, it is sufficient and convenient to define the $B$-matrix
of $\sd$ to be $\tB$, like most references on cluster algebras.
However, one could replace $\tB$ by $(b_{ij})_{i,j\in I}$ in the
definition of a seed $\sd$. This latter definition will be important
if we want to consider the corresponding cluster Poisson algebra;
see \cite{gross2018canonical}\cite{kimura2022twist} for examples.

In examples, we will provide a full matrix $(b_{ij})_{i,j\in I}$
for a given seed $\sd$. The choice of entries $(b_{ij})_{i,j\in I_{\fv}}$
is not unique and will be unimportant for our purpose.

\end{rem}

By a quiver, we mean an oriented graph, whose oriented edges will
be called arrows. We always assume it it to be locally finite, i.e.,
each vertex is an end of finitely many arrows. It is often convenient
to draw a weighted (locally finite) quiver $\tQ$, called an ice quiver,
to represent the matrix $(b_{ij})_{i,j\in I}$:
\begin{itemize}
\item The set of vertices of $\tQ$ is $I$. The full subquiver on $I_{\ufv}$
is denote by $Q$. We often call $\tQ$ an ice quiver and $Q$ its
principal part.
\item For any $i,j\in I$ , if $b_{ij}>0$, we draw arrows with positive
weights from $i$ to $j$, such that their total weight is $b_{ij}$. 
\end{itemize}
Note that $b_{ji}=-\frac{\sym_{i}}{\sym_{j}}b_{ij}$. Only the total
weight of arrows matters for us: if two quivers have the same total
weight of arrows for any given pair $(i,j)$, we view them as the
same. Conversely, any such a weighted quiver determines a matrix $(b_{ij})$.

In our figures, solid (resp. dashed) arrows represent arrows of weight
$1$ (resp. $\frac{1}{2}$), and circular (resp. rectangular) nodes
represent unfrozen (resp. frozen) vertices.

\begin{eg}[{A seed from $\C[SL_3]$}]\label{eg:SL3-seed}

We consider the following seed arising from the cluster structure
on $\C[SL_{3}]$. $I=\{k\in\Z|k\in[-1,6]\}$, $I_{\fv}=\{-1,0,5,6\}$,
all $\sym_{i}=1$, and $(b_{ij})_{-1\leq i,j\leq6}=\left(\begin{array}{cccccccc}
0 & -\frac{1}{2} & 1 & 0 & 0 & 0 & 0 & 0\\
\frac{1}{2} & 0 & -1 & 0 & 1 & 0 & 0 & 0\\
-1 & 1 & 0 & -1 & 0 & 0 & 0 & 0\\
0 & 0 & 1 & 0 & -1 & 0 & 1 & 0\\
0 & -1 & 0 & 1 & 0 & -1 & 0 & 0\\
0 & 0 & 0 & 0 & 1 & 0 & -1 & \frac{1}{2}\\
0 & 0 & 0 & -1 & 0 & 1 & 0 & -1\\
0 & 0 & 0 & 0 & 0 & -\frac{1}{2} & 1 & 0
\end{array}\right)$. The associated ice quiver is depicted in Figure \ref{fig:SL3}.

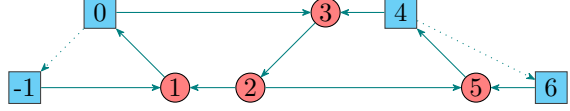
\begin{figure}
\caption{The ice quiver for a seed of $\C[SL_{3}]$}
\label{fig:SL3}

\begin{tikzpicture}
 [node distance=48pt,on grid,>={Stealth[length=4pt,round]},bend angle=45, inner sep=0pt]
\node[frozen] (q-1) at (-4.5,-0.5) {-1};
\node[frozen] (q0) at (-3.5,0.5) {0};
\node[unfrozen] (q1) at (-2.5,-0.5) {1};
\node[unfrozen] (q2) at (-1.5,-0.5) {2};
\node[unfrozen] (q3) at (-0.5,0.5) {3};
\node[frozen] (q4) at (0.5,0.5) {4};
\node[unfrozen] (q5) at (1.5,-0.5) {5};
\node[frozen] (q6) at (2.5,-0.5) {6};
\draw[->,teal]  (q-1) edge (q1);
\draw[->,teal]  (q2) edge (q1);
\draw[->,teal]  (q2) edge (q5);
\draw[->,teal]  (q6) edge (q5);
\draw[->,teal]  (q5) edge (q4);
\draw[->,teal]  (q4) edge (q3);
\draw[->,teal]   (q0) edge (q3);
\draw[->,teal]  (q3) edge (q2);
\draw[->,teal]   (q1) edge (q0);
\draw[->,dotted,teal]  (q0) edge (q-1);
\draw[->,dotted,teal]  (q4) edge (q6);
\end{tikzpicture}

\end{figure}

\end{eg}

\begin{eg}[An infinite seed of type $A_1$]\label{eg:GHL-A1-inf}

Now consider an infinite seed appearing in \cite{geiss2024representations},
such that $I=I_{\ufv}=\Z$, all $\sym_{i}=1$, and $(b_{ij})_{i,j\in\Z}$
is given by 
\begin{align}
(b_{ij})_{i,j\in\Z} & =\left(\begin{array}{ccccccccc}
\cdots & \cdots\\
\cdots & 0 & 1\\
 & -1 & 0 & 1\\
 &  & -1 & 0 & 1\\
 &  &  & -1 & \red{0} & -1\\
 &  &  &  & 1 & 0 & 1\\
 &  &  &  &  & -1 & 0 & 1\\
 &  &  &  &  &  & -1 & 0 & \cdots\\
 &  &  &  &  &  &  & \cdots & \cdots
\end{array}\right)\label{eq:GHL-A1-inf-B}
\end{align}
is determined by the quiver in Figure \ref{fig:A1-GHL} (the central
entry in \eqref{eq:GHL-A1-inf-B} is the $b_{00}=0$). Note that almost
all arrows point to the left except the one between the nodes $0$
and $1$.

\begin{figure}
\caption{The quiver for an infinite seed of type $A_{1}$ in \cite{geiss2024representations}}
\label{fig:A1-GHL}

\begin{tikzpicture}
 [scale=0.8,node distance=48pt,on grid,>={Stealth[length=4pt,round]},bend angle=45, inner sep=0pt]

\node (q4) at (-8,0) {$\cdots$};
\node[unfrozen] (q3) at (-6,0) {3};
\node[unfrozen] (q2) at (-4,0) {2};
\node[unfrozen] (q1) at (-2,0) {1};
\node[unfrozen,fill=pink!50] (q0) at (0,0) {0};
\node[unfrozen,fill=pink!50] (q-1) at (2,0) {-1};
\node[unfrozen,fill=pink!50] (q-2) at (4,0) {-2};
\node[unfrozen,fill=pink!50](q-3) at (6,0) {-3};
\node(q-4) at (8,0) {$\cdots$};

\draw[->,teal]   (q-4) edge (q-3);
\draw[->,teal]   (q-3) edge (q-2);
\draw[->,teal]   (q-2) edge (q-1);
\draw[->,teal]   (q-1) edge (q0);
\draw[->,teal]   (q1) edge (q0);
\draw[->,teal]   (q1) edge (q0);
\draw[->,teal]   (q1) edge (q2);
\draw[->,teal]   (q2) edge (q3);
\draw[->,teal]   (q3) edge (q4);
\end{tikzpicture}
\end{figure}

\end{eg}

If $\sigma$ is a permutation on $I$ such that $\sigma I_{\ufv}=I_{\ufv}$,
we define the seed $\sd'=\sigma\sd$ such that $I'=I$, $I'_{\ufv}=I$,
$\sym'_{\sigma i}=\sym_{i}$, and $b'_{\sigma i,\sigma j}=b_{i,j}$.
In the quantum case, we endow $\sigma\sd$ with the compatible Poisson
structure $\sigma\lambda$ whose associated matrix is $\sigma\Lambda$.

\subsubsection*{Torus algebras}

Following \cite{gross2018canonical}, define $\Mcirc:=\oplus_{i\in I}\Z f_{i}\simeq\Z^{I}$
and $N:=\oplus_{i\in I}\Z e_{i}\simeq\Z^{I}$, where $f_{i}$ and
$e_{i}$ are the standard basis vectors. Define the linear map $p^{*}:N\rightarrow\Mcirc$
such that 
\begin{align*}
p^{*}(e_{j}) & =\sum b_{ij}f_{i}.
\end{align*}
Therefore, we have $p^{*}n=\tB n$ for $n\in N$ as coordinate (column)
vectors. We will denote $N_{\ufv}:=\oplus_{k\in I_{\ufv}}\Z e_{k}$,
$N^{\oplus}:=\oplus_{k\in I_{\ufv}}\N e_{k}$, and $M^{\oplus}:=p^{*}N^{\oplus}$.
We could view $\Mcirc\otimes\R$ and $N\otimes\R$ as dual spaces
by introduce the pairing $\langle f_{i},e_{j}\rangle=\delta_{ij}\frac{\sym_{i}}{\sym}$
for any chosen constant $\sym\neq0$. When the set $\{\sym_{i},i\in I\}$
is finite, one could choose $\sym$ to be the least common multiplier
of $\{\sym_{i},i\in I\}$ as in \cite[Section 2.1]{qin2019bases}
(our $\sym_{i}$ were denoted by $d_{i}^{\vee}$ and $\frac{\sym}{\sym_{i}}$
were denoted by $d_{i}$ in loc. cit., so that the conventions were
compatible with \cite{gross2018canonical}).

Unless otherwise specified, we will make the following assumption.

\begin{asm}[Injectivity assumption]

The map $p^{*}$ is injective on $N_{\ufv}$.

\end{asm}

Let $\kk$ denote a commutative ring, which we choose to be $\Z$
for the classical case and $\Z[v^{\pm}]$ for the quantum case unless
otherwise specified. Define the torus algebra $\LP$ to be the Laurent
polynomial ring $\kk[x_{i}^{\pm}]_{i\in I}$, whose commutative multiplication
is denoted by $\cdot$. We naturally identify $\LP$ with the group
ring $\kk[\Mcirc]:=\oplus_{m\in\Mcirc}\kk x^{m}$ such that $\prod_{i\in I}x_{i}^{m_{i}}=x^{m}$
for any $m=(m_{i})_{i\in I}=\sum m_{i}f_{i}\in\Mcirc$. Similarly,
we identify $\kk[y_{i}^{\pm}]_{i\in I}$ with $\kk[N]:=\oplus\kk y^{n}$
such that $\prod_{i}y_{i}^{n_{i}}=y^{n}$ for $n=(n_{i})_{i\in I}=\sum n_{i}e_{i}\in N$.
The linear map $p^{*}$ induces a homomorphism $p^{*}:\kk[N]\rightarrow\kk[M]$
such that 
\begin{align*}
p^{*}(y^{n}) & :=x^{p^{*}(n)}=x^{\tB n}.
\end{align*}
In particular, $p^{*}(y_{k})=x^{\col_{k}\tB}$, $\forall k\in I_{\ufv}$.
When there is no confusion, we might denote the Laurent monomial $x^{p^{*}(n)}$
by $y^{n}$ for simplicity.

\begin{defn}

The cluster monomials of $\sd$ are $x^{m}$, $m\in\N^{I}$. Its localized
cluster monomials are $x^{m}$, $m\in\N^{I_{\ufv}}\oplus\Z^{I_{\fv}}$,
and its frozen factors are $x^{p},p\in\Z^{I_{\fv}}$.

\end{defn}

Let $\widehat{\kk[N^{\oplus}]}$ denote the completion of $\kk[N^{\oplus}]$
with respect to the maximal ideal generated by $y_{k}$, $k\in I_{\ufv}$,
i.e., it is the ring of formal series $\kk\llbracket y_{k}\rrbracket_{k\in I_{\ufv}}$.
Since $p^{*}$ is injective, all $p^{*}(y_{k})$ for $k\in I_{\ufv}$
generate a maximal ideal for $\kk[M^{\oplus}]$. Let $\widehat{\kk[M^{\oplus}]}$
denote the correspond formal completion. We have $p^{*}:\widehat{\kk[N^{\oplus}]}\simeq\widehat{\kk[M^{\oplus}]}$.
Finally, we define the following formal completion:
\begin{align*}
\hLP & :=\LP\otimes_{\kk[M^{\oplus}]}\widehat{\kk[M^{\oplus}]}.
\end{align*}

We will also denote $\bLP:=\kk[x_{j}]_{j\in I_{\fv}}[x_{k}]_{k\in I_{\ufv}}$,
called the partially compactified torus algebra.

\subsubsection*{Quantization}

A compatible Poisson structure (or a compatible Poisson form) for
$\sd$ is a $\Z$-valued skew-symmetric bilinear form $\lambda$ on
$\Mcirc$, such that there exist $\diag_{k}>0$ for $k\in I_{\ufv}$
and 
\begin{align*}
\lambda(f_{i},p^{*}e_{k}) & =-\delta_{ik}\diag_{k},\ \forall i\in I,k\in I_{\ufv}.
\end{align*}
If such a $\lambda$ exists, we define the skew-symmetric matrix $\Lambda=(\Lambda_{ij})_{i,j\in I}$,
called the quantization matrix, such that $\Lambda_{ij}=\lambda(f_{i},f_{j})$.
The pair $(\Lambda,\tB)$ will be called a compatible pair as in \cite{BerensteinZelevinsky05}.
Note that $p^{*}$ must be injective if $\lambda$ exists. By \cite{gekhtman2003cluster,GekhtmanShapiroVainshtein05},
when $|I|<\infty$ and $p^{*}$ is injective, a compatible $\lambda$
must exist.

A quantum seed is a seed $\sd$ endowed with a compatible Poisson
structure $\lambda$. When we work with a quantum seed $\sd$ at the
quantum level $\kk=\Z[v^{\pm}]$, we will endow the torus algebra
$\LP$ with the $v$-twisted product $*$ such that
\begin{align*}
x^{m}*x^{n} & :=v^{\lambda(m,n)}x^{m+n},\ \forall m,n\in\Mcirc.
\end{align*}
And we will use $*$ as the multiplication for the algebra $\LP$
unless otherwise specified. 

We let $\cF$ denote the skew field of fractions of $\LP$; see \cite{BerensteinZelevinsky05}.
We also introduce the bar involution $\overline{(\ )}$ on $\LP$,
which is the $\Z$-linear map such that $\overline{v^{\alpha}x^{m}}=v^{-\alpha}x^{m}$.
It is an anti-involution.

\begin{eg}\label{eg:GHL-A1-inf-quantize}

It is a good exercise to find a compatible Poisson form $\lambda$
for the infinite seed in Example \ref{eg:GHL-A1-inf}. For convenience
of the reader, let us provide an explicit solution: the quantization
matrix associated with $\lambda$ is $\Lambda=(\Lambda_{ij})_{i,j\in\Z}$:
\begin{align}
\Lambda= & \left(\begin{array}{ccccccccccc}
 &  &  &  &  & \cdots\\
 & 0 & 1 & 0 & 1 & 0 & -1 & 0 & -1 & 0\\
 & -1 & 0 & 1 & 0 & 1 & 0 & -1 & 0 & -1\\
 & 0 & -1 & 0 & 1 & 0 & -1 & 0 & -1 & 0\\
 & -1 & 0 & -1 & 0 & 1 & 0 & -1 & 0 & -1\\
\cdots & 0 & -1 & 0 & -1 & \red{0} & -1 & 0 & -1 & 0 & \cdots\\
 & 1 & 0 & 1 & 0 & 1 & 0 & 1 & 0 & 1\\
 & 0 & 1 & 0 & 1 & 0 & -1 & 0 & 1 & 0\\
 & 1 & 0 & 1 & 0 & 1 & 0 & -1 & 0 & 1\\
 & 0 & 1 & 0 & 1 & 0 & -1 & 0 & -1 & 0\\
 &  &  &  &  & \cdots
\end{array}\right).\label{eq:GHL-A1-inf-Lambda}
\end{align}
The central entry in \eqref{eq:GHL-A1-inf-Lambda} is $\Lambda_{00}=0$.
The nonzero entries of $\Lambda$ are given by
\begin{align*}
\Lambda_{i,i+2d+1} & =\begin{cases}
1 & i<i+2d+1\leq0\\
-1 & \text{else}
\end{cases} & \forall i<0,d\in\Z\\
\Lambda_{0,2d+1} & =-1 & \forall d\in\Z\\
\Lambda_{i,i-2d-1} & =\begin{cases}
-1 & 0<i-2d-1<i\\
1 & \text{else}
\end{cases} & \forall i>0,d\in\Z
\end{align*}
This matrix $\Lambda$ can be computed inductively as in Example \ref{eg:uniqe-A1-quantization}.

Note that the matrix $\tB=B$ in \eqref{eq:GHL-A1-inf-B} is given
by
\begin{align*}
\col_{i}\tB & =e_{i-1}-e_{i+1} & \forall i\neq0,1\\
\col_{0}\tB & =e_{-1}+e_{1}\\
\col_{1}\tB & =-e_{0}-e_{2}
\end{align*}
It is straightforward to verify that $\Lambda\tB$ is an $\Z\times\Z$
diagonal matrix whose diagonal entries are $-2$, i.e., all $\diag_{k}=2$.
Therefore, $\lambda$ is a compatible Poisson pair. We will see in
Example \ref{eg:uniqe-A1-quantization} that how to determine all
solutions. 

We refer the reader to \cite{paganelli2025quantum} for another approach
to $\lambda$ with representation theoretic interpretations.

\end{eg}

By changing the sign in $\sd$, we obtain its opposite seed $\sd^{\op}:=(I,I_{\ufv},(\sym_{i})_{i\in I},-\tB,(x_{i})_{i\in I})$.
When $\sd$ is endowed with a compatible Poisson form $\lambda$,
we endow $\sd^{\op}$ with $-\lambda$. There is a canonical $\kk$-linear
anti-isomorphism (\cite[(3.1)]{qin2020analog})
\begin{align}
\iota & :\LP(\sd)\simeq\LP(\sd^{\op}),\ \iota(x^{m}):=x^{m}.\label{eq:iota-opposite}
\end{align}

\begin{rem}\label{rem:opposite-seed}

At the classical level, seeds for the same cluster algebras in different
literature sometimes differ by a sign, because $\iota$ is an isomorphism
and the sign difference is negligible. However, $\iota$ is not an
isomorphism at the quantum level. We will pay attention to distinguish
$\sd$ and $\sd^{\op}$, though it usually does not bring new information
in practice.

\end{rem}

\subsubsection*{Mutations}

Denote $[\ ]_{+}:=\max(\ ,0)_{+}$. Let $\sd$ be a given seed. For
any unfrozen vertex $k\in I_{\ufv}$, we have an operation $\mu_{k}$
called mutation, which produces a new seed $\sd'=(I,I_{\ufv},(d_{i})_{i\in I},\tB',(x'_{i})_{i\in I})$
such that $x'_{i}$ are indeterminates, and $\tB'=(b_{ij}')_{i\in I,j\in I_{\ufv}}$.
Here, we define a new matrix $(b_{ij}')_{i,j\in I}:=\mu_{k}(b_{ij})_{i,j\in I}$
by
\begin{align*}
b'_{ij} & =\begin{cases}
-b_{ij} & \text{if \ensuremath{i=k} or \ensuremath{j=k}}\\
b_{ij}+[b_{ik}]_{+}[b_{kj}]_{+}-[-b_{ik}]_{+}[-b_{kj}]_{+} & \text{else}
\end{cases}.
\end{align*}
We further define maps associated with mutations: $\psi^{M}:(\Mcirc)'\simeq\Mcirc$
and $\psi^{N}:N'\simeq N$ such that
\begin{align*}
\psi^{M}(f'_{i}) & =\begin{cases}
-f_{k}+\sum_{j\in I}[-b_{jk}]_{+}f_{j} & \text{if \ensuremath{i=k}}\\
f_{i} & \text{else}
\end{cases}
\end{align*}
\begin{align*}
\psi^{N}(e'_{i}) & =\begin{cases}
-e_{k} & \text{if \ensuremath{i=k}}\\
e_{i}+[b_{ki}]_{+}e_{k} & \text{else}
\end{cases}
\end{align*}
These constructions are associated with positive sign mutations; see
\cite[Section 2.7]{kimura2022twist} for a general definition.

When $\sd$ is a quantum seed, its Poisson structure $\lambda$ induces
a compatible Poisson structure on $\sd'$ via $\psi^{M}$, still denoted
$\lambda$:
\begin{align}
\lambda(m,m') & :=\lambda(\psi^{M}m,\psi^{M}m'),\ \forall m,m'\in(\Mcirc)'.\label{eq:mutation-Lambda}
\end{align}

We further introduce a $\kk$-algebra isomorphism $\mu_{k}^{*}:\cF'\simeq\cF$
between the associated skew fields of fractions, called the mutation
map, such that
\begin{align}
\mu_{k}^{*}(x'_{i}) & =\begin{cases}
x^{-f_{k}+\sum_{j\in I}[-b_{jk}]_{+}f_{j}}+x^{-f_{k}+\sum_{i\in I}[b_{ik}]_{+}f_{i}} & \text{if \ensuremath{i=k}}\\
x_{i} & \text{else}
\end{cases}.\label{eq:exchange-relation}
\end{align}
We will identify $\cF'$ and $\cF$ via $\mu_{k}^{*}$ and omit the
symbol $\mu_{k}^{*}$. Note that the mutation is an involution on
seeds and on skew fields; see \cite{fomin2002cluster,BerensteinZelevinsky05}.

In general, for any finite sequence of unfrozen vertices $\uk=(k_{1},\ldots,k_{r})$,
we have the mutation sequence $\seq:=\seq_{\uk}:=\mu_{k_{r}}\ldots\mu_{k_{1}}$
(read from right to left). For $\sd':=\seq\sd$, the mutation map
$\seq^{*}:\cF'\simeq\cF$ is defined as the composition $\mu_{k_{1}}^{*}\cdots\mu_{k_{r}}^{*}$.

\subsubsection*{Cluster algebras}

Let $\sd$ be any given (classical or quantum) initial seed. Let $\Delta^{+}:=\Delta_{\sd}^{+}$
denote the set of seed $\seq\sd$, where $\seq$ is any finite mutation
sequence. Identify $\cF$ and $\cF(\seq\sd)$ via $\seq^{*}$ as before.
Note that $x_{j}(\seq\sd)=x_{j}(\sd)$ for $j\in I_{\fv}$, which
we can simply denote by $x_{j}$.

We define the following versions of cluster algebras associated to
$\sd$:
\begin{itemize}
\item The partially compactified ordinary cluster algebra is 
\begin{align*}
\bClAlg(\sd) & :=\kk[x_{i}(\sd')]_{i\in I,\sd'\in\Delta^{+}},
\end{align*}
where $x_{i}(\sd')$ denote the cluster variables of the seed $\sd'$.
\item The partially compactified upper cluster algebra is 
\begin{align*}
\bUpClAlg(\sd) & :=\cap_{\sd'\in\Delta^{+}}\bLP(\sd').
\end{align*}
\item The localized ordinary cluster algebra is 
\begin{align*}
\clAlg(\sd) & :=\bClAlg(\sd)[x_{j}^{-1}]_{j\in I_{\fv}}=\kk[x_{j}^{\pm}]_{j\in I_{\fv}}[x_{k}(\sd')]_{k\in I_{\ufv},\sd'\in\Delta^{+}}.
\end{align*}
\item The localized upper cluster algebra is 
\begin{align*}
\upClAlg(\sd) & :=\bUpClAlg(\sd)[x_{j}^{-1}]_{j\in I_{\fv}}=\cap_{\sd'\in\Delta^{+}}\LP(\sd').
\end{align*}
\end{itemize}
The rank of a cluster algebra is defined to be that of $\sd$. 

We will drop the symbol $\sd$ when we do not want to emphasize our
choice of the initial seed $\sd$. By the Laurent phenomenon \cite{fomin2002cluster}\cite{BerensteinZelevinsky05},
we have $\bClAlg\subset\bUpClAlg$ and $\clAlg\subset\upClAlg$.

We define the frozen torus algebra to be $\cRing=\kk[x_{j}^{\pm}]_{j\in I_{\fv}}$.
Then $\clAlg$ and $\upClAlg$ are $\cRing$-algebras.

\subsection{Tropical points}

For the case $\sd'=\mu_{k}\sd$, we introduce the tropical mutation
$\phi_{\sd',\sd}:\Mcirc\rightarrow(\Mcirc)'$, which is a piecewise
linear map such that $m=(m_{i})_{i\in I}$ is sent to $m'=(m'_{i})_{i\in I}$
given by
\begin{align*}
m'_{i}:= & \begin{cases}
-m_{k} & \text{if \ensuremath{i=k}}\\
m_{i}+[b_{ik}]_{+}[m_{k}]_{+}-[-b_{ik}]_{+}[-m_{k}]_{+} & \text{else}
\end{cases}.
\end{align*}
It can viewed as the change of coordinates associated to the change
of basis $\psi^{M}$; see \cite[Lemma A.1.2]{qin2019bases}.

Now, consider any mutation sequence $\seq=\seq_{\uk}$, where $\uk=(k_{1},\ldots,k_{r})$.
We will denote $\sd_{0}=\sd$ and $\sd_{t}:=\mu_{t}\cdots\mu_{2}\mu_{1}\sd$
for $t\in[1,r]$. For $\sd'=\seq\sd$, define the tropical mutation
$\phi_{\sd',\sd}:\Mcirc\rightarrow(\Mcirc)'$ as composition $\phi_{\sd_{r},\sd_{r-1}}\cdots\phi_{\sd_{1},\sd_{0}}$.
Note that $\phi_{\sd',\sd}$ can be realized as the tropicalization
of $\seq^{*}$; see \cite{gross2013birational}. In particular, the
tropical mutation $\phi_{\sd',\sd}$ only depends on the mutation
map $\seq^{*}$, not on the specific sequence $\uk$.

As in \cite{qin2019bases}, we define the set of tropical point as
\begin{align*}
\Mtrop & :=(\sqcup_{\sd'\in\Delta^{+}}\Mcirc(\sd'))/\sim,
\end{align*}
where $\sim$ is the equivalence relations that $m\in\Mcirc(\sd)$
and $m'\in\Mcirc(\sd')$ are equivalent if and only if $m'=\phi_{\sd',\sd}(m)$.
The equivalence class of $m'\in\Mcirc(\sd')$ is called a tropical
point, denoted $[m']$. Conversely, $m'$ is called the representative
of $[m']\in\Mtrop$ in $(\Mcirc)'$. We refer the reader to \cite{gross2018canonical}
for the geometric definition of tropical points.

\subsection{Degrees and pointedness}

Let $\sd$ be any given seed. For the construction below, we need
$\sd$ to satisfy the Injectivity Assumption.

\begin{defn}

An element $z\in\hLP$ is said to be $m$-pointed for some $m\in\Mcirc$
if it takes the form 
\begin{align*}
z & =x^{m}\cdot(1+\sum_{0\neq n\in N^{\oplus}}c_{n}p^{*}(y^{n})).
\end{align*}
If $z'=v^{\alpha}z$ for some $\alpha\in\Z$, we say $z'$ has degree
$m$ and its (degree) normalization is $z$ with respect to $\sd$,
denoted $\deg^{\sd}(z')=m$ and $[z']^{\sd}:=z$. We call $F_{z}:=1+\sum c_{n}y^{n}\in\kk\llb N^{\oplus}\rrb$
the $F$-function of $z$. $F_{z}$ is called the $F$-polynomial
if $F_{z}\in\kk[N^{\oplus}]$.

Similarly, a formal sum $z$ of Laurent monomials is called $h$-copointed
for some $h\in\Mcirc$ if it takes the form
\begin{align*}
z & =x^{h}\cdot(1+\sum_{0\neq n\in N^{\oplus}}c_{n}p^{*}(y^{-n})).
\end{align*}
If $z'=v^{\alpha}z$ for some $\alpha\in\Z$, we say $z'$ has codegree
$m$, denoted $\codeg^{\sd}(z')=m$.

If $z\in\LP$ has degree $m$ and codegree $h$, we say it has bidegree
$(m,h)$. We further define its support dimension to be $\suppDim z\in N^{\oplus}$
such that $h=m+p^{*}(\suppDim z)$. If it is further $m$-pointed
and $h$-copointed, we say it is bipointed at bidegree $(m,h)$.

\end{defn}

The notions of degree, codegree, and normalization depends on the
seed $\sd$. When the context is clear, we will drop the symbol $\sd$. 

We introduce the following partial order on $\Mcirc$, such that $m$
becomes the highest (or leading) degree of an $m$-pointed element,
and $h$ becomes the lowest (or trailing) degree of an $h$-copointed
element.

\begin{defn}[{Dominance order \cite{qin2017triangular} \cite[Proof of Proposition 4.3]{cerulli2015caldero}}]

For any $m,m'\in\Mcirc(\sd)$, we say $m$ dominates $m'$, denoted
$m'\prec_{\sd}m$, if $m'=m+p^{*}(n)$ for some $0\neq n\in N^{\oplus}$.

\end{defn}

We will denote $\LP_{\preceq m}:=\oplus_{m'\preceq m}\kk x^{m'}$
and $\hLP_{\preceq m}:=\LP_{\preceq m}\otimes_{\kk[M^{\oplus}]}\widehat{\kk[M^{\oplus}]}$.

Assume $\sd'=\seq\sd$ for some mutation sequence $\seq$. We say
$z\in\LP(\sd)\cap\LP(\sd')$ is compatibly pointed at $\sd,\sd'$,
if $z$ is $m$-pointed in $\LP(\sd)$, $m'$-pointed in $\LP(\sd')$,
such that $[m]=[m']$. For any tropical point $[m]\in\Mtrop$, we
say $z\in\upClAlg=\cap_{\sd'\in\Delta^{+}}\LP(\sd')$ is $[m]$-pointed
if, for any $\sd'$, it is $m'$-pointed in $\LP(\sd')$ such that
$[m']=[m]$.

\begin{defn}

A subset $Z$ of $\hLP(\sd)$ is said to be $\Theta$-pointed for
some $\Theta\subset\Mcirc(\sd)$ if $Z$ takes the form $\{Z_{m}|m\in\Theta\}$
such that $Z_{m}$ are $m$-pointed.

A subset $Z$ of $\upClAlg$ is said to be $\Theta$-pointed for some
$\Theta\subset\Mtrop$ if $Z$ takes the form $\{Z_{[m]}|[m]\in\Theta\}$
such that $Z_{[m]}$ are $[m]$-pointed.

\end{defn}

By tropical properties, we refer to the behavior of the degrees of
pointed elements under mutations.

\subsection{Injective reachability}

We will often need the following assumption. 

\begin{asm}[Injective reachability]

The seed $\sd$ is injective reachable, i.e., there exists a sequence
of mutation $\Sigma$ and a permutation $\sigma$ on $I_{\ufv}$ such
that $\pr_{I_{\ufv}}\deg^{\sd}x_{\sigma k}(\Sigma\sd)=-f_{k}$, $\forall k\in I_{\ufv}$.

\end{asm}

The mutation sequence $\Sigma$ is called a green to red sequence
in \cite{keller2011cluster}; see \cite[Proposition 2.3.3]{qin2019bases}.
The assumption is known to be satisfied by most known cluster algebras
arising from representation theory or (higher) Teichmüller theory.
The exceptions are those arising from once-puncture closed surfaces;
see \cite{FominShapiroThurston08,fomin2018cluster}\cite{yurikusa2020density}.

The seed $\Sigma\sd$ will be denoted by $\sd[1]$. Recursively, we
could define $\sd[d]$ for $d\in\Z$, such that $\sd[0]=\sd$ and
$\sd[d+1]=(\sigma^{d}\Sigma)\sd[d]$, where $\sigma\seq_{(k_{1},\ldots,k_{r})}:=\seq_{(\sigma k_{1},\ldots,\sigma k_{r})}$.
It is known that if one seed is injective reachable, then so are all
seeds of the cluster algebra; see \cite{qin2017triangular}\cite{muller2015existence}.

The combination of injective reachability and tropical properties
produce some useful results.

\begin{lem}[{\cite[Lemma 3.4.12]{qin2019bases}}]\label{lem:tropical-to-cluster-monomial}

Assume that $z\in\upClAlg$ and $z'$ is a localized cluster monomial
of a seed $\sd'$. If $z$ and $z'$ are both $m$-pointed in $\LP(\sd)$,
and they are compatibly pointed at $\sd,\sd',\sd'[-1]$, then $z=z'$.

\end{lem}

\begin{thm}[{\cite[Theorem 4.3.1]{qin2019bases}}]\label{thm:tropical-to-basis}

Let $\base$ denote a $\Mcirc$-pointed subset of $\upClAlg(\sd)$.
If $\base$ is compatibly pointed at the seeds appearing along a mutation
sequence from $\sd$ to $\sd[-1]$, then $\base$ is a basis of $\upClAlg(\sd)$.

\end{thm}

Note that the basis $\base$ in Theorem \ref{thm:tropical-to-basis}
must contain all localized cluster monomials by Lemma \ref{lem:tropical-to-cluster-monomial}.

\subsection{Valuations, partial compactification, and optimization}

Take any seed $\sd$. We do not assume $p^{*}$ is injective for the
moment. For any $j\in I$, we introduce the valuation $\nu_{j}$ on
$\cF(\sd)$ such that for any $z=x_{j}^{d}*P*Q^{-1}$, where $P,Q\in\kk[x_{i}]_{i\in I}$
are not divisible by $x_{j}$, we set $\nu_{j}(z):=d$. We say $z$
is regular at $x_{j}=0$ if $\nu_{j}(z)\geq0$.

\begin{lem}[{\cite[Lemma 2.12]{qin2023analogs}}]

For $\sd'=\mu_{k}\sd$, the valuation $\nu_{j}'$ defined on $\cF(\sd')$
satisfy $\nu'_{j}(z')=\nu_{j}(\mu_{k}^{*}z)$ for any $z'\in\cF(\sd')$.

\end{lem}

Therefore, we can simply denote the valuation by $\nu_{j}$ without
specifying it is defined on which $\cF(\sd)$. 

It follows from definition that we have $\bUpClAlg=\{z\in\upClAlg|\nu_{j}(z)\geq0,\forall j\in I_{\fv}\}$
for upper cluster algebras. However, we only know $\bClAlg\subset\{z\in\clAlg|\nu_{j}(z)\geq0,\forall j\in I_{\fv}\}$
for ordinary cluster algebras. 

\begin{rem}

By using valuation, $\bUpClAlg$ is easy to understand once we know
$\upClAlg$, while our inability to characterize the algebra $\bClAlg$
is often an obstruction to prove desired properties for $\bClAlg$.
For example, while we could show some quantized coordinate rings are
$\bUpClAlg$, we still could not show that they are $\bClAlg$ without
additional knowledge \cite{qin2024infinite}\cite{qin2025partially}.

\end{rem}

Let us make the Injectivity Assumption from now on. 

\begin{defn}[{\cite[Definition 2.16]{qin2023analogs}}]\label{def:dom-trop-point}

Choose any $\Mtrop$-pointed basis $\base=\{\base_{[m]}|[m]\in\Mtrop\}$
for $\upClAlg$, where $\base_{[m]}$ are $[m]$-pointed. We define
the corresponding set of dominant tropical points to be $\domMtrop:=\{[m]\in\Mtrop|\nu_{j}(\base_{[m]})\geq0,\forall j\in I_{\fv}\}$.
Let $\domM$ denote the set of representatives of $\domMtrop$ in
$\Mcirc$.

\end{defn}

In definition \ref{def:dom-trop-point}, we have $\base\cap\bUpClAlg=\{\base_{[m]}|[m]\in\domMtrop\}$.
It is desirable to ask if $\base\cap\bUpClAlg$ is a basis for $\bUpClAlg$.

\begin{defn}

We say $j\in I_{\fv}$ is optimized in a seed $\sd$ if $b_{jk}\geq0$
for all $k\in I_{\ufv}$. We say $j\in I_{\fv}$ can be optimized
if it is optimized in some seed in $\Delta^{+}$.

\end{defn}

When $j$ is optimized in $\sd$, we have $\nu_{j}(z)=m_{j}$ for
any $m$-pointed element $z$ in $\LP(\sd)$. We further deduce the
following pleasant result.

\begin{prop}[{\cite[Proposition 2.15]{qin2023analogs}\cite[Proposition 9.7]{gross2018canonical}}]\label{prop:optimized-compactified-basis}

Assume all frozen vertices can be optimized. If $\base$ is a $\Mtrop$-pointed
basis of $\upClAlg$, then $\base\cap\bUpClAlg$ is a basis for $\bUpClAlg$.

\end{prop}

\begin{rem}

Note that, when all frozen vertices $j\in I_{\fv}$ can be optimized,
$\domMtrop$ is independent of the choice of the basis $\base$ and
can be calculated explicitly. More precisely, assume $j$ is optimized
in $\sd_{j}$, the set of representatives of $\domMtrop$ in $\Mcirc(\sd)$
is given by $\domM(\sd)=\{m\in\Mcirc(\sd)|(\phi_{\sd_{j},\sd}(m))_{j}\geq0,\forall j\in I_{\fv}\}$;
see \cite[Lemma 2.17]{qin2023analogs}.

\end{rem}

\begin{rem}

Many cluster algebras possess a well-behaved $\Mcirc$-pointed basis
$\vartheta$ called the theta basis \cite{gross2018canonical}. By
a recent work \cite{cheung2025valuative}, the statement of Proposition
\ref{prop:optimized-compactified-basis} is true for the theta basis
$\vartheta$ even without the assumption that the frozen vertices
can be optimized. 

It is desirable to know if results analogous to that of \cite{cheung2025valuative}
hold for other good bases $\base$, for example, for the common triangular
basis (Section \ref{sec:Common-triangular-bases}) or the generic
basis (see \cite{plamondon2013generic}\cite{qin2019bases} for definitions
and results). Note that $\nu_{j}(z)$ could be interpreted as the
value of a tropical function on the Laurent degrees of $z$ \cite{gross2018canonical}.
So analogous results hold for other $\Mcirc$-pointed bases $\base$,
if we can prove the following conjecture:

For any $m\in\Mcirc$, the $m$-pointed basis elements $\vartheta_{m}$
and $\base_{m}$ have the same Newton polytope.

\end{rem}

\section{Seeds associated with signed words}\label{sec:Seeds-associated-with-words}

We will construct seeds from signed words following \cite{BerensteinFominZelevinsky05}\cite{shen2021cluster}.
They will be useful to describe cluster algebras arising from Lie
theory. 

\subsection{Seeds construction}

Let $C=(C_{ab})_{a,b\in J}$ denote a generalized Cartan matrix, where
$J\subset\N_{>0}$ is a finite set of index. Denote its symmetrizers
by $D_{a}$, $a\in J$, i.e., we have $D_{a}C_{ab}=D_{b}C_{ba}$,
$\forall a,b\in J$. 

For any $r,s\in\Z$. We use $[r,s]$ to denote the integer interval
$\{k\in\Z|r\leq k\leq s\}$. A word $\ueta=(\eta_{k})_{k\in[a,b]}$
is a sequence of elements in $J$. A signed word $\ubi=(\bi_{k})_{k\in[a,b]}$
is a sequence of elements in $\pm J$, where we denote $\bi_{k}=\varepsilon_{k}i_{k}$
for $i_{k}\in J$ and $\varepsilon_{k}=\pm1$. We denote their length
$l(\ueta)=l(\ubi)=l([r,s]):=r-s+1$. When $[r,s]=\emptyset$, $\ubi_{[r,s]}$
is an empty sequence.

For any subinterval $[r',s']\subset[r,s]$, we use $\bi_{[r',s']}$
to denote the subsequence $(\bi_{r'},\bi_{r'+1},\ldots,\bi_{s'})$.
By abuse of notation, we will write 
\begin{align*}
\ubi & =(\bi_{r'},\ldots,\bi_{r'-1},\ubi_{[r',s']},\bi_{s'+1},\ldots,\bi_{s}).
\end{align*}
The opposite word $\ubi^{\op}$ is defined such that $(\ubi^{\op})_{k}:=\bi_{s-k+r}$,
$\forall k\in[r,s]$. For any $N\in\N$, we let $\ubi^{N}$ denote
the sequence $(\ubi,\ubi,\ldots,\ubi)$ without specifying its index.

We associate with $\ubi$ a classical seed $\dsd=\dsd(\ubi)$ as follows.

Denote $\dot{I}:=[r,s]$. The set of vertex for $\dsd$ is $I:=\ddot{I}:=[r-|J|,r-1]\sqcup\dot{I}=[r-|J|,s]$.
Choose a Coxeter word $\uc=(c_{1},\ldots,c_{|J|})$, i.e. each element
of $J$ appears in $\uc$ exactly once. We define the skew-symmetrizers
to be $\sym_{k}:=\begin{cases}
D_{c_{k-r+|J|+1}} & k\in[r-|J|,r-1]\\
D_{i_{k}} & k\in\dot{I}
\end{cases}$. Define $[1]:\ddot{I}\rightarrow\ddot{I}\cup\{+\infty\}$ such that
\begin{align}
k[1] & =\min(\{j\in\ddot{I}|j>k,i_{j}=i_{k}\}\cup\{+\infty\}).\label{eq:shift-letter}
\end{align}
Denote $k[0]=k$ for $k\in\ddot{I}$. Denote $k[-1]:=k'$ if $k=k'[1]\in\ddot{I}$.
For $d\in\Z$ and $\varepsilon\in\{\pm1\}$, if $k[d]$ and $(k[d])[\varepsilon]$
are both well defined, we define $k[d+\varepsilon]:=(k[d])[\varepsilon]$.

Define $\dot{I}_{\fv}:=\{k\in\dot{I}|k[1]=+\infty\}$ and the set
of frozen vertices $\ddot{I}_{\fv}:=[r-|J|,r-1]\sqcup\dot{I}_{\fv}$.
Then the set of unfrozen vertex is $\dot{I}_{\ufv}:=\ddot{I}_{\ufv}:=\{k\in\dot{I}|k[1]\in\dot{I}\}$.

Following \cite{BerensteinFominZelevinsky05}\cite{shen2021cluster},
define the exchange matrix $\ddot{B}=(b_{jk})_{j\in\ddot{I},k\in\ddot{I}_{\ufv}}$
associated to $\ubi$ by 
\begin{align*}
b_{jk}= & \begin{cases}
\varepsilon_{k} & \text{if }k=j[1]\\
-\varepsilon_{j} & \text{if }j=k[1]\\
\varepsilon_{k}C_{i_{j},i_{k}} & \text{if }\varepsilon_{j[1]}=\varepsilon_{k},\ j<k<j[1]<k[1]\\
\varepsilon_{k}C_{i_{j},i_{k}} & \text{if }\varepsilon_{k}=-\varepsilon_{k[1]},\ j<k<k[1]<j[1]\\
-\varepsilon_{j}C_{i_{j},i_{k}} & \text{if }\varepsilon_{k[1]}=\varepsilon_{j},\ k<j<k[1]<j[1]\\
-\varepsilon_{j}C_{i_{j},i_{k}} & \text{if }\varepsilon_{j}=-\varepsilon_{j[1]},\ k<j<j[1]<k[1]\\
0 & \text{otherwiese}
\end{cases}
\end{align*}
We refer the reader to \cite[Section 6.1]{qin2023analogs} for an
intuitive construction of $\ddot{B}$ using triangulated trapezoid
following \cite{shen2021cluster}, which also leads to the construction
of half weight entries $b_{ij}$ for $i,j\in\ddot{I}_{\fv}$.

The seed $\rsd$ is obtained from $\dsd$ by restricting to the subset
of vertices $\dot{I}$. More precisely, $I(\rsd):=\dot{I}$, $I(\rsd)_{\ufv}:=\dot{I}_{\ufv}$,
$\tB(\rsd):=(b_{jk})_{j\in\dot{I},k\in\dot{I}_{\ufv}}$, and the skew-symmetrizers
are $\sym_{k}=D_{i_{k}}$ for $k\in\dot{I}$.

The frozen vertices of $\dsd$ could be optimized; see \cite[Proposition 6.16]{qin2023analogs}.
By \cite{BerensteinZelevinsky05}, there is a compatible Poisson structure
for $\dsd$. The same is true for $\rsd$ by \cite[Lemma 6.4]{qin2023analogs}. 

These definitions of $[r,s]$, (signed) words, and the associated
(quantum) seeds could be naturally extend to the case where $r=-\infty$
and/or $s=+\infty$.

\subsection{Operations on signed words}\label{subsec:Operations-on-signed}

\subsubsection*{Signed words and group elements}

Recall that the monoid $\Br^{+}$ of positive braids associated to
$C$ is the monoid generated by the generators $\{\sigma_{a}|a\in J\}$
subject to the following relations:
\begin{align*}
\sigma_{a}\sigma_{b}=\sigma_{b}\sigma_{a} & \text{ if }C_{ab}=0\\
\sigma_{a}\sigma_{b}\sigma_{a}=\sigma_{b}\sigma_{a}\sigma_{a} & \text{ if }C_{ab}C_{ba}=1\\
(\sigma_{a}\sigma_{b})^{2}=(\sigma_{b}\sigma_{a})^{2} & \text{ if }C_{ab}C_{ba}=2\\
(\sigma_{a}\sigma_{b})^{3}=(\sigma_{b}\sigma_{a})^{3} & \text{ if }C_{ab}C_{ba}=3
\end{align*}

The generalized Braid group $\Br$ is constructed from $\Br^{+}$
by including the formal inverses $\sigma_{a}^{\pm}$, where $a\in J$.
The associated Coxeter group $W$ is the quotient group of $\Br$
by the relation $\sigma_{a}^{2}=e$ for $a\in J$, where $e$ denotes
the identity element. For any $\beta\in\Br$, we denote its image
in $W$ by $[\beta]$. In particular, we have the simple reflections
$s_{a}=[\sigma_{a}]$.

For any word $\ueta=(\eta_{k})_{k\in[r_{1},r_{2}]}$, denote its length
by $l(\ueta):=l([r_{1},r_{2}])=r_{2}-r_{1}+1$. We let $\beta_{\ueta}$
denote the positive braid $\sigma_{\eta_{r_{1}}}\cdots\sigma_{\eta_{r_{2}}}$
and $w_{\ueta}$ denote the Coxeter group element $s_{\eta_{r_{1}}}\cdots s_{\eta_{r_{2}}}$,
which equals $[\beta_{\ueta}]$. For any $w\in W$, define its length
to be $l(w):=\min\{r|s_{\eta_{1}}\cdots s_{\eta_{r}}=w,\ \eta_{i}\in J,\ \forall i\in[1,r]\}$.
$\ueta$ is called a reduced word if $l(\ueta)=l(w_{\ueta})$. When
$W$ is finite, we use $w_{0}$ to denote its longest element.

\subsubsection*{Signed words Operations}

Let $\ubi=(\bi_{k})_{k\in[r,s]}$ denote a signed word. Let $\sd$
denote $\rsd$ or $\dsd.$ Following \cite{shen2021cluster}, we could
construct a new signed word $\ubi'=(\bi'_{k})_{k\in[r,s]}$ and thus
new classical seeds $\sd(\ubi')$ by one of the following operations.
\begin{itemize}
\item (Left reflection) Define $\ubi'=(-\bi_{r},\bi_{[r+1,s]})$. In this
case, we have $\rsd(\ubi')=\rsd(\ubi)$.
\item (Flips) When $k\in[r,s-1]$ and $\varepsilon_{k}=-\varepsilon_{k+1}$,
a flip at $k$ gives a new signed word $\ubi'=(\ubi_{[r,k-1]},\bi_{k+1},\bi_{k},\ubi_{[k+2,s]})$.
Let $\sigma_{k,k+1}$ denote the transposition $(k,k+1)$ on $\Z$.
In this case, we have $\sd(\ubi')=\sigma_{k,k+1}\sd(\ubi)$ if $i_{k}\neq i_{k+1}$
and $\sd(\ubi')=\mu_{k}\sd(\ubi)$ if $i_{k}=i_{k+1}$.
\item (Braid moves) If $\ubi_{[j,k]}=\varepsilon\ueta$ for some word $\ueta$
in $J$, $\varepsilon\in\{\pm1\}$, and $\beta_{\ueta}=\beta_{\ueta'}$
in $\Br^{+}$, we could construct $\ubi'=(\ubi_{[r,j-1]},\varepsilon\ueta',\ubi_{[k+1,s]})$.
Then there exists a sequence of mutation $\seq_{\ubi',\ubi}$ on $\{k'\in[j,k]|k'[1]\leq k\}$
and a permutation $\sigma_{\ubi',\ubi}$ on $[j,k]$ such that $\sd(\ubi')=\sigma_{\ubi',\ubi}\seq_{\ubi',\ubi}\sd(\ubi)$.
\end{itemize}
Let $\ubi'$ be a signed word obtained form $\ubi$ by a finite sequence
of operations above. If we endow $\sd(\ubi)$ with a compatible Poisson
structure, then we could obtain an induced compatible Poisson structure
on $\sd(\ubi')$ by using mutations step by step (see \eqref{eq:mutation-Lambda}).
Moreover, this induced Poisson structure is independent of the choice
of the operation sequence; see \cite[Lemma 3.2]{qin2024infinite}.

\subsection{Injective reachable sequences and fundamental variables}\label{subsec:Injective-reachable-fundamental-var}

For any signed word $\ubi=(\bi_{k})_{k\in[r,s]}$, we introduce the
following notation for any $k\in[r,s]$:
\begin{align*}
o_{+}(k) & :=|\{k'|k'>k,k'\in[r,s],i_{k}=i_{k'}\}|\\
k^{\max} & :=k[o_{+}(k)]
\end{align*}
Recall that denote $k^{\max}[1]=+\infty$ by \eqref{eq:shift-letter}.
We will work with $\dot{I}=[r,s]$. Correspondingly, we introduce
the following
\begin{align*}
o_{-}(k) & :=|\{k'|k'<k,k'\in[r,s],i_{k}=i_{k'}\}|\\
k^{\min} & :=k[o_{-}(k)]\\
k^{\min}[-1] & :=-\infty.
\end{align*}
We will denote $f_{-\infty}=0$. 

Let $\ueta=(\eta_{k})_{k\in[r,s]}$ denote a length $l$ word in $J$.
We claim that the seed $\rsd:=\rsd(\ueta)$ is injective reachable
(\cite[Theorem 4.1]{shen2021cluster}). To see this, denote $\seq_{[j,k]}:=\mu_{k}\cdots\mu_{j[1]}\mu_{j}$
for $j,k\in[1,l]$ (read from right to left) such that $\eta_{j}=\eta_{k}$.
Define $\Sigma_{k}:=\seq_{[k^{\min},k^{\min}[o_{+}(k)-1]]}$, and
$\Sigma:=\Sigma_{s}\cdots\Sigma_{r+1}\Sigma_{r}$. In particular,
$\Sigma_{k}$ is an empty sequence if $k=k^{\max}$. Let $\sigma$
denote the permutation on $[1,l]$ such that $\sigma(k)=k^{\max}[o_{-}(k)-1]$. 

\begin{lem}[{\cite[Lemma 8.4]{qin2023analogs}}]\label{lem:interval-var-deg}

The cluster variables of the seeds appearing along the mutation sequence
$\Sigma$ starting from $\rsd$ have the degrees $f_{k}-f_{j[-1]}$
for $[j,k]\subset[r,s]$ such that $\eta_{j}=\eta_{k}$.

In particular, we have $x_{\sigma k}(\Sigma\rsd)=-f_{k}+f_{k^{\max}}$,
$\forall k\in I_{\ufv}$. Consequently, $\rsd$ satisfies the injective-reachability
assumption.

\end{lem}

\begin{defn}

For any $[j,k]\subset[r,s]$ such that $\eta_{j}=\eta_{k}$, let $W_{[j,k]}$
denote the cluster variable with degree $f_{k}-f_{j[-1]}$ (see Lemma
\ref{lem:interval-var-deg}). $W_{[j,k]}$ is called an interval variable.
$W_{k}:=W_{[k,k]}$ is called a fundamental variable.

\end{defn}

\begin{rem}

In special cases, $W_{k}$ correspond to fundamental modules of quantum
affine algebras; hence the name (see Section \ref{subsec:Cluster-algebras-from-q-aff-alg}).

\end{rem}

\subsection{Examples}

\begin{eg}[{Seeds arising from $U_q(\hat{sl}_4)$-representations \cite{HernandezLeclerc09}}]

\label{eg:HL-A2-inf} Choose the Cartan matrix $C=\left(\begin{array}{ccc}
2 & -1 & 0\\
-1 & 2 & -1\\
0 & -1 & 2
\end{array}\right)$ and the Coxeter word $\uc=(1,2,3)$. Take the words $\ubj=(\bj_{k})_{k\in[1,8]}=(\uc^{4})$,
and $\ubi=(\bi_{k})_{k\in[1,\infty)}=(\uc^{\infty})$. The ice quivers
for $\rsd(\ubj)$ and $\rsd(\ubi)$ are depicted in Figure \ref{fig:quiver-A3-copy4}
and Figure \ref{fig:quiver-A3-inf}, respectively.

\begin{figure}
\caption{The ice quiver for $\rsd(\ubj)$, $\ubj=(1,2,3)^{4}$.}

\label{fig:quiver-A3-copy4}

\begin{tikzpicture}
 [scale=1.5,node distance=48pt,on grid,>={Stealth[length=4pt,round]},bend angle=45, inner sep=0pt]

 \node[unfrozen](v1) at (0,0){1};    
\node[unfrozen](v2) at (-12pt,-24pt){2};
\node[unfrozen](v3) at (-24pt,-48pt){3};
\node[unfrozen](v4) at (-24pt,0pt){4};
\node[unfrozen](v5) at (-36pt,-24pt){5};
\node[unfrozen](v6) at (-48pt,-48pt){6};
\node[unfrozen](v7) at (-48pt,0pt){7};
\node[unfrozen](v8) at (-60pt,-24pt){8};
\node[unfrozen](v9) at (-72pt,-48pt){9};
\node[frozen](v10) at (-72pt,0pt){10};
\node[frozen](v11) at (-84pt,-24pt){11};
\node[frozen](v12) at (-96pt,-48pt){12};

\draw[->] (v2) -- (v1);
\draw[->] (v5) -- (v4);
\draw[->] (v8) -- (v7);
\draw[->] (v1) -- (v4); 
\draw[->] (v2) -- (v5); 
\draw[->] (v4) -- (v7); 
\draw[->] (v5) -- (v8); 
\draw[->] (v4) -- (v2); 
\draw[->] (v7) -- (v5);

\draw[->,dashed] (v12) -- (v11);
\draw[->] (v11)--(v9);
\draw[->] (v9)--(v8);
\draw[->] (v8)--(v6);
\draw[->] (v6)--(v5);
\draw[->] (v5)--(v3);
\draw[->] (v3)--(v2);

\draw[->,dashed] (v11) -- (v10);
\draw[->] (v10)--(v8);

\draw[->] (v3)--(v6);
\draw[->] (v6)--(v9);
\draw[->] (v9)--(v12);

\draw[->] (v8) -- (v11);
\draw[->] (v7) -- (v10);

\end{tikzpicture}
\end{figure}

\begin{figure}
\caption{The ice quiver for $\rsd(\ubi)$, $\ubi=(1,2,3)^{\infty}$.}

\label{fig:quiver-A3-inf}

\begin{tikzpicture}
 [scale=1.5,node distance=48pt,on grid,>={Stealth[length=4pt,round]},bend angle=45, inner sep=0pt]

 \node[unfrozen](v1) at (0,0){1};    
\node[unfrozen](v2) at (-12pt,-24pt){2};
\node[unfrozen](v3) at (-24pt,-48pt){3};
\node[unfrozen](v4) at (-24pt,0pt){4};
\node[unfrozen](v5) at (-36pt,-24pt){5};
\node[unfrozen](v6) at (-48pt,-48pt){6};
\node[unfrozen](v7) at (-48pt,0pt){7};
\node[unfrozen](v8) at (-60pt,-24pt){8};
\node[unfrozen](v9) at (-72pt,-48pt){9};
\node[unfrozen](v10) at (-72pt,0pt){10};
\node[unfrozen](v11) at (-84pt,-24pt){11};
\node[unfrozen](v12) at (-96pt,-48pt){12};
\node(v13) at (-96pt,0pt){$\cdots$};
\node(v14) at (-108pt,-24pt){$\cdots$};
\node(v15) at (-120pt,-48pt){$\cdots$};

\draw[->] (v2) -- (v1);
\draw[->] (v5) -- (v4);
\draw[->] (v8) -- (v7);
\draw[->] (v1) -- (v4); 
\draw[->] (v2) -- (v5); 
\draw[->] (v4) -- (v7); 
\draw[->] (v5) -- (v8); 
\draw[->] (v4) -- (v2); 
\draw[->] (v7) -- (v5);

\draw[->] (v12) -- (v11);
\draw[->] (v11)--(v9);
\draw[->] (v9)--(v8);
\draw[->] (v8)--(v6);
\draw[->] (v6)--(v5);
\draw[->] (v5)--(v3);
\draw[->] (v3)--(v2);

\draw[->] (v11) -- (v10);
\draw[->] (v10)--(v8);

\draw[->] (v3)--(v6);
\draw[->] (v6)--(v9);
\draw[->] (v9)--(v12);

\draw[->] (v8) -- (v11);
\draw[->] (v7) -- (v10);

\draw[->] (v12) -- (v15);
\draw[->] (v11) -- (v14);
\draw[->] (v10) -- (v13);

\draw[->] (v14) -- (v12);
\draw[->] (v13) -- (v11);

\end{tikzpicture}
\end{figure}

\end{eg}

\begin{eg}[{A seed for $\C[SL_3]$ \cite{BerensteinFominZelevinsky05}}]

Choose the Cartan matrix $C=\left(\begin{array}{cc}
2 & -1\\
-1 & 2
\end{array}\right)$ and the signed word $\ubi=(1,-1,2,-2,1,-1)$. The seed $\dsd(\ubi)$
is given Example \ref{eg:SL3-seed} (Figure \ref{fig:SL3}). 

\end{eg}

\begin{eg}[{An infinite seed of type $A_2$ \cite{geiss2024representations}}]

Choose the Cartan matrix $C=\left(\begin{array}{cc}
2 & -1\\
-1 & 2
\end{array}\right)$, $\uc=(1,2)$, and the signed word $\ubi=(\uc^{\infty},1,-1,2,-2,1,-1,(\uc^{\op})^{\infty})$$=(\ldots,1,2,1,2,1,-1,2,-2,1,-1,2,1,2,1,\ldots)$.
The ice quiver $\tQ$ for the seed $\dsd(\ubi)$ is given Figure \ref{fig:quiver-inf-A2}.

\begin{figure}

\caption{The ice quiver for $\dsd(\ubi)$, $\ubi=((1,2)^{\infty},1,-1,2,-2,1,-1,(2,1)^{\infty})$}

\label{fig:quiver-inf-A2}

\begin{tikzpicture}
 [node distance=48pt,on grid,>={Stealth[length=4pt,round]},bend angle=45, inner sep=0pt]
\node[unfrozen] (q-1) at (-4,-0.5) {-1};
\node[unfrozen] (q0) at (-3.5,0.5) {0};
\node[unfrozen] (q1) at (-2.5,-0.5) {1};
\node[unfrozen] (q2) at (-1.5,-0.5) {2};
\node[unfrozen] (q3) at (-0.5,0.5) {3};
\node[unfrozen] (q4) at (0.5,0.5) {4};
\node[unfrozen] (q5) at (1.5,-0.5) {5};
\node[unfrozen] (q6) at (2.5,-0.5) {6};

\node[unfrozen] (q-2) at (-4.5,0.5) {-2};
\node[unfrozen] (q-3) at (-5,-0.5) {-3};
\node (q-4) at (-5.5,0.5) {$\cdots$};
\node(q-5) at (-6,-0.5) {$\cdots$};
\node[unfrozen] (q7) at (3,0.5) {7};
\node[unfrozen] (q8) at (3.5,-0.5) {8};
\node[unfrozen](q9) at (4,0.5) {9};
\node[unfrozen](q10) at (4.5,-0.5) {10};
\node(q11) at (5,0.5) {$\cdots$};
\node(q12) at (5.5,-0.5) {$\cdots$};

\draw[->]  (q-1) edge (q1);
\draw[->]  (q2) edge (q1);
\draw[->]  (q2) edge (q5);
\draw[->]  (q6) edge (q5);
\draw[->]  (q5) edge (q4);
\draw[->]  (q4) edge (q3);
\draw[->]   (q0) edge (q3);
\draw[->]  (q3) edge (q2);
\draw[->]   (q1) edge (q0);
\draw[->]  (q0) edge (q-1);

\draw [->] (q-2) edge (q0);
\draw [->] (q-3) edge (q-1);
\draw [->] (q-1) edge (q-2);
\draw [->] (q-4) edge (q-2);
\draw [->] (q-5) edge (q-3);
\draw [->] (q-3) edge (q-4);
\draw [->] (q-2) edge (q-3);
\draw [->] (q4) edge (q7);
\draw [->] (q7) edge (q6);
\draw [->] (q6) edge (q8);
\draw [->] (q8) edge (q7);
\draw [->] (q7) edge (q9);
\draw [->] (q9) edge (q8);
\draw [->] (q8) edge (q10);
\draw [->] (q9) edge (q11);
\draw [->] (q10) edge (q12);
\draw [->] (q11) edge (q10);
\draw [->] (q10) edge (q9);
\end{tikzpicture}

\end{figure}

\end{eg}

\begin{eg}[{A seed for $\C[SL_2]$}]

Choose the Cartan matrix $C=\left(\begin{array}{c}
2\end{array}\right)$ and the signed word $\ubi=(\bi_{k})_{k\in[0,1]}=(1,-1)$. Then $\dsd(\ubi)$
is a seed for the cluster algebra $\C[SL_{2}]$; see Section \ref{subsec:Cluster-structures-on-G}.
The corresponding ice quiver is given Figure \ref{fig:quiver-SL2}.

\begin{figure}
\caption{The ice quiver for a seed of $\ensuremath{\C[SL_{2}]}$.}

\label{fig:quiver-SL2}

\begin{tikzpicture}
 [node distance=48pt,on grid,>={Stealth[length=4pt,round]},bend angle=45, inner sep=0pt]

\node[frozen] (q1) at (-2,0) {1};
\node[unfrozen] (q0) at (0,0) {0};
\node[frozen] (q-1) at (2,0) {-1};

\draw[->]   (q-1) edge (q0);
\draw[->]   (q1) edge (q0);
\draw[->]   (q1) edge (q0);

\end{tikzpicture}
\end{figure}

\end{eg}

\begin{eg}[{An infinite seed of type $A_1$ \cite{geiss2024representations}}]

Choose the Cartan matrix $C=\left(\begin{array}{c}
2\end{array}\right)$ and the infinite signed word $\ubi=(\ldots,1,1,-1,1,1,\ldots)$ where
$\bi_{1}=-1$. The seed $\dsd(\ubi)$ is given in Example \ref{eg:GHL-A1-inf}
(Figure \ref{fig:A1-GHL}).

\end{eg}

\section{Cluster operations}\label{sec:Cluster-operations}

Following \cite{qin2023analogs}, we introduce some cluster operations
that propagate structures from one cluster algebra to another cluster
algebra whose seeds are closely related. These operations are crucial
in the approach of \cite{qin2023analogs} to Theorem \ref{thm:Lie-tri-basis}
(results for cluster algebras from Lie theory). 

This section is technical. Readers may skip it without loss of continuity
in the subsequent sections.

\subsection{Freezing}

Let $\sd=(I,I_{\ufv},(\sym_{i})_{i\in I},\tB,(x_{i})_{i\in I})$ denote
a seed. Let $(b_{ij})_{i,j\in I}$ be the associated matrix for its
chosen ice quiver. Choose any $F\subset I_{\ufv}$. By freezing the
vertices in $F$, we obtain a new seed $\sd':=\frz_{F}\sd=(I,I_{\ufv}\backslash F,(\sym_{i})_{i\in I},\tB,(x_{i})_{i\in I})$.
The matrix $(b_{ij})_{i,j\in I}$ is kept unchanged. When $\sd$ is
a quantum seed endowed with the compatible Poisson structure $\lambda$,
we endow $\sd'$ with the same compatible Poisson structure $\lambda$. 

We have $\Mcirc(\sd')=\Mcirc(\sd)$ but $N^{\oplus}(\sd')\subset\N^{\oplus}(\sd)$.
Note that the set of tropical points $\Mtrop$ for $\upClAlg(\sd)$
and the set of tropical points $(\Mtrop)'$ for $\upClAlg(\sd')$
are different, although they are both in bijection with $\Mcirc(\sd)=\Mcirc(\sd')\simeq\Z^{I}$.

Assume $p^{*}$ is injective from now on. Note that $\preceq_{\sd}$
implies $\preceq_{\sd'}$, which are both denoted by $\preceq$. We
define the freezing operator $\frz_{F,m}^{\sd}$ to be a linear map
from $\hLP_{\preceq m}(\sd)$ to $\hLP_{\preceq m}(\sd')$ such that
\begin{align*}
\frz_{F,m}^{\sd}(\sum_{n\in N^{\oplus}}b_{n}x^{m+p^{*}(n)}) & :=\sum_{0\neq n\in N^{\oplus},\supp n\cap F=\emptyset}b_{n}x^{m+p^{*}(n)},\ \forall b_{n}\in\kk.
\end{align*}
Note that, if $z$ is $m$-pointed, $\frz_{F,m}^{\sd}(z)$ is $m$-pointed
too.

When the context is clear, we might abbreviate the symbols $\sd$,
$F$, and/or $m$ in $\frz_{F,m}^{\sd}$. For a subset $Z$ of $\hLP(\sd)$
whose elements have unique $\prec$-maximal degrees, we denote $\frz Z:=\{\frz z|z\in Z\}$.
By freezing $F'\in I$, we mean freezing $F:=F'\cap I_{\ufv}$.

Note that $\frz$ is not a linear map from $\hLP(\sd)$ to $\hLP(\sd')$.
Nevertheless, Lemma \ref{lem:freezing-basic-properties} states that
it still has some compatibility with algebra operations.

\begin{lem}[{\cite[Lemma 3.6]{qin2023analogs}}]\label{lem:freezing-basic-properties}

(1) Assume that we have $m'\preceq m$ in $\Mcirc$ such that $m'=m+p^{*}(n)$,
$z\in\hLP_{\preceq m'}(\sd)\subset\hLP_{\preceq m}(\sd)$, then we
have 
\begin{align*}
\frz_{m}z & =\begin{cases}
\frz_{m'}z & \text{if }\supp n\cap F=\emptyset\\
0 & \text{else}
\end{cases}.
\end{align*}

(2) For any $m_{1},m_{2}\in\Mcirc$ and $z_{1}\in\hLP_{\preceq m_{1}}$,
$z_{2}\in\hLP_{\preceq m_{2}}$, we have 
\begin{align*}
\frz_{m_{1}}(z_{1})*\frz_{m_{2}}(z_{2}) & =\frz_{m_{1}+m_{2}}(z_{1}*z_{2}).
\end{align*}

\end{lem}

\begin{prop}[{\cite[Proposition 3.8]{qin2023analogs}}]\label{prop:freeze-to-up}

If $z\in\LP_{\preceq m}(\sd)\cap\upClAlg(\sd)$, then $\frz_{F,m}z\in\upClAlg(\frz_{F}\sd)$.

\end{prop}

We have the following important result, which states that the freezing
operator $\frz_{F}$ is compatible with mutations at $k\notin F$.

\begin{prop}[{\cite[Lemma 3.9]{qin2023analogs}}]\label{prop:freeze-compatible-mutation}

Assume $k\notin F$ and take $\sd=\mu_{k}\sd'$. Then for any $z\in\LP(\sd)\cap(\mu_{k}^{*})^{-1}\LP(\sd')$
such that $z$ has a unique $\prec_{\sd}$-maximal degree $m$ in
$\LP(\sd)$ and $\mu_{k}^{*}z$ has a unique $\prec_{\sd'}$-maximal
degree $m'$ in $\LP(\sd')$, then we have $\frz_{F,m'}^{\sd'}(\mu_{k}^{*}z)=\mu_{k}^{*}\frz_{F,m}^{\sd}(z)$.

\end{prop}

Denote $\sd'=\frz_{F}\sd$. Combining Proposition \ref{prop:freeze-to-up}
and Proposition \ref{prop:freeze-compatible-mutation}, we obtain
the following result.

\begin{thm}[{\cite[Theorem 3.10]{qin2023analogs}}]\label{thm:freeze-compatily-pointed}

If $z\in\upClAlg(\sd)$ is compatibly pointed at all seeds in $\Delta_{\sd}^{+}$,
then $\frz_{F}z\in\upClAlg(\sd')$ is compatibly pointed at all seeds
in $\Delta_{\sd'}^{+}$.

\end{thm}

Theorem \ref{thm:freeze-compatily-pointed} and Theorem \ref{thm:tropical-to-basis}
implies the following.

\begin{thm}[{\cite[Theorem 3.11]{qin2023analogs}}]

If $\base$ is a $\Mtrop$-pointed basis of $\upClAlg(\sd)$ and $\frz_{F}\sd$
is injective reachable, then $\frz_{F}\base$ is a $(\Mtrop)'$-pointed
basis of $\upClAlg(\sd')$.

\end{thm}

We also have the following general result.

\begin{thm}[{\cite[Theorem 3.12]{qin2023analogs}\cite[Theorem 4.3.1]{qin2023freezing}}]\label{thm:freeze-cluster-monomial}

If $z$ is a localized cluster monomial of $\upClAlg(\sd)$, then
$\frz_{F}z$ is a localized cluster monomial of $\upClAlg(\frz_{F}\sd)$.

\end{thm}

In general, it is difficult to know if $z\in\upClAlg$ also belongs
to $\clAlg$. Though this happens if $z$ has the following property.

\begin{defn}[{\cite[Definition 3.14]{qin2023analogs}}]

Assume there is a finite decomposition $z=\sum b_{i}z_{i}$, such
that $b_{i}\in\kk$, $z_{i}$ have distinct degrees in $\LP(\sd)$,
and $z_{i}$ are products of cluster variables and frozen factors,
then we say $z$ has a nice cluster decomposition in $\sd$.

\end{defn}

\begin{lem}\cite[Lemma 3.15]{qin2023analogs}

For any $z\in\LP_{\preceq m}(\sd)$, if $z$ has a nice cluster decomposition
in $\sd$, then $\frz_{m}z\in\clAlg(\frz_{F}\sd)$.

\end{lem}

\begin{cor}[{\cite[Corollary 3.16]{qin2023analogs}}]

If $\base$ is a $\Mtrop$-pointed basis of $\upClAlg(\sd)$, such
that its elements have nice cluster decomposition in $\sd$, and if
$\frz_{F}\sd$ is injective reachable, then $\frz_{F}\base$ is a
$(\Mtrop)'$-pointed basis of $\upClAlg(\frz_{F}\sd)$ and its elements
have nice cluster decomposition. In particular, $\clAlg(\frz_{F}\sd)=\upClAlg(\frz_{F}\sd)$.

\end{cor}

\begin{eg}\label{eg:freezing}

Consider the seed $\sd=\rsd(\ubi)$ where $\ubi=(\bi_{k})_{k\in[1,6]}=(1,2)^{3}$.
The ice quiver for $\sd$ is given in Figure \ref{fig:HL-A2-copy3}.
The element $z=x_{2}^{-1}\cdot x_{6}\cdot(1+y_{2}+y_{1}\cdot y_{2}+y_{2}\cdot y_{4}+y_{1}\cdot y_{2}\cdot y_{4}+y_{1}\cdot y_{2}\cdot y_{3}\cdot y_{4})$
is the $(-f_{2}+f_{6})$-pointed common triangular basis element.
It can be identified with the (truncated) $(q,t)$-character of a
Kirillov-Reishetikhin module (Section \ref{subsec:Cluster-algebras-from-q-aff-alg}).

Now freeze $\sd$ at $F=\{4\}$. The ice quiver for $\frz_{F}\sd$
is depicted in Figure \ref{fig:HL-A2-copy3-freeze}. We obtain $\frz_{F}z=x_{2}^{-1}\cdot x_{6}\cdot(1+y_{2}+y_{1}\cdot y_{2})$
be setting $y_{4}=0$. 

\begin{figure}
\caption{The ice quiver for $\rsd(\ubi)$, $\ubi=(1,2)^{3}$}
\label{fig:HL-A2-copy3}

\begin{tikzpicture} [scale=2,node distance=48pt,on grid,>={Stealth[round]},bend angle=45,      pre/.style={<-,shorten <=1pt,>={Stealth[round]},semithick},    post/.style={->,shorten >=1pt,>={Stealth[round]},semithick},  unfrozen/.style= {circle,inner sep=1pt,minimum size=12pt,draw=black!100,fill=red!100},  frozen/.style={rectangle,inner sep=1pt,minimum size=12pt,draw=black!75,fill=cyan!100},   point/.style= {circle,inner sep=1pt,minimum size=5pt,draw=black!100,fill=black!100},   boundary/.style={-,draw=cyan},   internal/.style={-,draw=red},    every label/.style= {black}]
\node[unfrozen](v1) at (0,0){1}; \node[unfrozen](v2) at (-10pt,-20pt){2}; \node[unfrozen](v3) at (-20pt,0pt){3}; \node[unfrozen](v4) at (-30pt,-20pt){4}; \node[frozen](v5) at (-40pt,0pt){5}; \node[frozen](v6) at (-50pt,-20pt){6};
\draw[->,dashed] (v6) -- (v5); \draw[->] (v5) -- (v4);  \draw[->] (v4) -- (v3);  \draw[->] (v3) -- (v2);  \draw[->] (v2) -- (v1);
\draw[->] (v1) -- (v3);  \draw[->] (v3) -- (v5); 
\draw[->] (v2) -- (v4);  \draw[->] (v4) -- (v6); 
\end{tikzpicture}

\end{figure}

\begin{figure}
\caption{The ice quiver for $\frz_{\{4\}}\rsd(\ubi)$, $\ubi=(1,2)^{3}$}
\label{fig:HL-A2-copy3-freeze}

\begin{tikzpicture} [scale=2,node distance=48pt,on grid,>={Stealth[round]},bend angle=45,      pre/.style={<-,shorten <=1pt,>={Stealth[round]},semithick},    post/.style={->,shorten >=1pt,>={Stealth[round]},semithick},  unfrozen/.style= {circle,inner sep=1pt,minimum size=12pt,draw=black!100,fill=red!100},  frozen/.style={rectangle,inner sep=1pt,minimum size=12pt,draw=black!75,fill=cyan!100},   point/.style= {circle,inner sep=1pt,minimum size=5pt,draw=black!100,fill=black!100},   boundary/.style={-,draw=cyan},   internal/.style={-,draw=red},    every label/.style= {black}]
\node[unfrozen](v1) at (0,0){1}; \node[unfrozen](v2) at (-10pt,-20pt){2}; \node[unfrozen](v3) at (-20pt,0pt){3}; \node[frozen](v4) at (-30pt,-20pt){4}; \node[frozen](v5) at (-40pt,0pt){5}; \node[frozen](v6) at (-50pt,-20pt){6};
\draw[->,dashed] (v6) -- (v5); \draw[->] (v5) -- (v4);  \draw[->] (v4) -- (v3);  \draw[->] (v3) -- (v2);  \draw[->] (v2) -- (v1);
\draw[->] (v1) -- (v3);  \draw[->] (v3) -- (v5); 
\draw[->] (v2) -- (v4);  \draw[->] (v4) -- (v6); 
\end{tikzpicture}
\end{figure}

\end{eg}

\subsection{Similarity and base change}

Recall that for a quantum seed $\sd$, we have $\lambda(f_{i},p^{*}e_{k})=-\delta_{ik}\diag_{k}$
for $k\in I_{\ufv}$, $\diag_{k}\in\N_{>0}$.

\begin{defn}[{\cite{Qin12}\cite{qin2017triangular}\cite[Definition 4.1]{qin2023analogs}}]

Let there be two seeds $\sd$ and $\sd'$, which are not necessarily
related by a sequence of mutations, and $\sigma$ a permutation of
$I'_{\ufv}$. We say $\sd$ and $\sd'$ are similar (up to $\sigma$)
if $b'_{\sigma i,\sigma j}=b_{ij}$ and $\sym'_{\sigma i}=\sym_{i}$,
$\forall i,j\in I_{\ufv}$. When $\sd$ and $\sd'$ are quantum seeds,
we further require that there exists some scaling factor $\rho\in\Q_{>0}$
such that $\diag_{\sigma k}'=\rho\diag_{k}$, $\forall k\in I_{\ufv}$.

\end{defn}

\begin{eg}[Principal coefficinet seed]\label{eg:principal-coeff-seed}

For any seed $\sd$, we can construct a new seed $\sd^{\prin}$, called
the seed of principal coefficients, as follows. Set $I_{\ufv}(\sd^{\prin}):=I_{\ufv}$,
$I_{\fv}(\sd^{\prin}):=\{k'|k\in I_{\ufv}\}$, and $\sym_{k'}(\sd^{\prin})=\sym_{k}(\sd^{\prin}):=\sym_{k}(\sd)$,
$\forall k\in I_{\ufv}$. Define $b_{ij}(\sd^{\prin})=b_{ij}(\sd)$
for $i,j\in I_{\ufv}$, and $b_{k'j}=\delta_{kj}$ for $j,k\in I_{\ufv}$.

There is a map $\var:\Mcirc(\sd^{\prin})\rightarrow\Mcirc(\sd)$ such
that
\begin{align}
\var(f_{i}) & =\begin{cases}
f_{i} & \text{if }i\in I_{\ufv}\\
\sum_{j\in I_{\fv}}b_{jk}f_{j} & \text{if }i=k'\in I_{\fv}(\sd^{\prin})
\end{cases}.\label{eq:principal-var-map}
\end{align}
It induces the $\kk$-linear map $\var$ from $\LP(\sd^{\prin})$
to $\LP(\sd)$ such that 
\begin{align*}
\var(x^{m}) & :=x^{\var(m)}.
\end{align*}
Note that $\var(x^{\tB(\sd^{\prin})e_{k}})=x^{\tB e_{k}}$, $\forall k\in I_{\ufv}$.

If $\sd$ is endowed with a compatible Poisson form $\lambda$, we
endow $\sd^{\prin}$ with the compatible Poisson form $\lambda(\sd^{\prin})$
such that 
\begin{align*}
\lambda(\sd^{\prin})(m,m') & :=\lambda(\var(m),\var(m')).
\end{align*}

The (classical or quantum) seed $\sd$ and $\sd'$ are similar with
$\rho=1$.

\end{eg}

In practice, it usually suffices to consider $\rho=1$. This happens
when two similar $\sd$ and $\sd'$ related by a mutation sequence,
for example, when $\sd'=\sd[1]$. Let us assume $\rho=1$ from now
on for simplicity. This assumption will greatly simply treatments
and statements, so that we do not need to extend $\kk$ to a base
ring $\kk'=\kk[v^{\frac{1}{r}}]$ for a large enough $r\in\N$ as
in \cite[Section 4]{qin2023analogs}.

Assume $\sigma$ and $\sigma'$ are similar (classical or quantum)
seeds and $\rho=1$. Replace $\sd'$ by $\sigma^{-1}\sd'$ if necessary,
we can assume $\sigma=1$ without loss of generality. We use $'$
to denote data associated with $\sd'$ as usual.

\begin{defn}[{\cite{Qin12}}]

Let there be given any $m$-pointed element $z=x^{m}\cdot(1+\sum_{0\neq n\in N^{\oplus}}b_{n}\cdot x^{\tB\cdot n})$
in $\hLP$, where $b_{n}\in\kk$. An $m'$-pointed element $z'\in\hLP'$
is said to be similar to $z$ if $\pr_{I_{\ufv}}m'=\pr_{I_{\ufv}}m$
and $z'$ takes the form $z'=(x')^{m'}\cdot(1+\sum_{0\neq n\in N^{\oplus}}b_{n}\cdot(x')^{\tB'\cdot n})$.

\end{defn}

The following correction technique is convenient for propagate algebra
relations from $\hLP$ to $\hLP'$. For $m\in\Mcirc$, let us define
the following set
\begin{align*}
\hV_{m} & :=x^{m}\cdot(\kk[p^{*}(N_{\ufv})]\otimes_{\kk[M^{\oplus}]}\widehat{\kk[M^{\oplus}]})\subset\hLP\\
V_{m} & :=x^{m}\cdot\kk[p^{*}(N_{\ufv})]\subset\LP
\end{align*}
For $m'\in(\Mcirc)'$, $\hV'_{m'}\subset\hLP'$ and $V'_{m'}\subset\LP'$
are defined similarly. Introduce the $\kk$-linear isomorphism $\var:=\var_{m',m}:\hV_{m}\rightarrow\hV'_{m'}$
such that 
\begin{align*}
\var(x^{m+\tB n}) & :=(x')^{m'+\tB'n}
\end{align*}
Then $\var$ sends a pointed element to a pointed element with the
same $F$-function.

\begin{thm}[{Correction technique \cite{Qin12}\cite{qin2017triangular}\cite[Theorem 4.6]{qin2023analogs}}]\label{thm:correction-technique}

Let there be given $m_{s}$-pointed elements $z_{s}\in\hLP$ and $m'_{s}$-pointed
elements $z'_{s}\in\hLP$ such that $z_{s}$ and $z'_{s}$ are similar,
where $s\in[1,r]$ for some $r\in\N$. We allow repetition in theses
degrees. Take $z\in\hV_{m}$ and $z'\in\hV'_{m'}$ such that $\pr_{I_{\ufv}}m=\pr_{I'_{\ufv}}m'$
and $z'=\var_{m',m}z$.
\begin{enumerate}
\item If $z=[z_{1}*\cdots*z_{r}]^{\sd}$, then $z'=p'\cdot[z'_{1}*\cdots*z'_{r}]^{\sd'}$
for the frozen factor $p'\in\cRing'$, such that $\deg^{\sd'}p'=m'-\sum_{s=1}^{r}m'_{s}$.
\item If $z=\sum_{s=1}^{r}b_{s}z_{s}$ for some $b_{s}\in\kk$, then $z=\sum_{J}b_{s}z_{s}$
for $J=\{s|b_{s}\neq0,\ z_{s}\in\hV_{m}\}$. Moreover, $z'=\sum_{s\in J}\var_{m',m}(b_{s}z_{s})=\sum_{s\in J}b_{s}p'_{s}z'_{s}$,
for the frozen factors $p'_{s}$ such that $\deg^{\sd'}p'_{s}=\deg^{\sd'}\var(z_{s})-m'_{s}$.
\end{enumerate}
\end{thm}

\begin{thm}[{\cite{FominZelevinsky07}\cite{Tran09}\cite{gross2018canonical}}]\label{thm:similar-cluster-var}

Let $\seq$ be any mutation sequence. For any $i\in I$, the cluster
variable $x_{i}(\seq\sd)\in\LP(\sd)$ and the cluster variable $x_{i}(\seq\sd')\in\LP(\sd')$
are similar.

\end{thm}

Let $\alg$ denote the ordinary cluster algebra $\clAlg$ or the upper
cluster algebra $\upClAlg$. The following statements are implied
by \cite[Lemma 4.7]{qin2023analogs}, Theorem \ref{thm:correction-technique},
and Theorem \ref{thm:similar-cluster-var}.

\begin{prop}

Let $z\in\LP$ and $z'\in\LP'$ be similar pointed elements. 

(1) If $z\in\alg$, then if $z'\in\alg'$.

(2) If $z$ has a nice cluster decomposition in $\sd$, then $z'$
has a nice cluster decomposition in $\sd'$.

\end{prop}

Now, let $\cZ$ be a subset of $\alg$ such that it a $\kk$-basis
\footnote{\cite[Section 4.2]{qin2023analogs} does not assume $\cZ$ to be a
basis. The corresponding statements change slightly.} and the following conditions are satisfied:
\begin{enumerate}
\item[(C1)]  The elements of $\cZ$ are pointed.
\item[(C2)]  $\cZ$ is closed under the commutative multiplication by frozen
factors:
\begin{align*}
x_{j}^{\pm}\cdot\cZ & \subset\cZ,\quad\forall j\in I_{\fv}.
\end{align*}
\end{enumerate}
Let $\cZ'$ denote the set of the elements in $\LP'$ that are similar
to elements of $\cZ$. Then $\cZ'$ is also closed under the commutative
multiplication by frozen factors.

\begin{prop}[{\cite[Lemmas 4.9, 4.10]{qin2023analogs}}]

The set $\cZ'$ is a $\kk$-basis of $\alg'$.

\end{prop}

Finally, let us give a diagrammatic description of the base change.
Assume that we have a $\kk$-algebra homomorphism $\var:\LP\rightarrow\LP'$
such that $\var(x_{k})=p_{k}\cdot x'_{k}$, $\var(x^{\col_{k}\tB})=(x')^{\col_{k}\tB'}$,
$\forall k\in I_{\ufv}$, and $\var$ sends any frozen factor to a
frozen factor. Such a map $\var$ is called a variation map following
\cite{kimura2022twist} (or a quasi homomorphism in \cite{fraser2016quasi}
in the classical case).

Note that $\LP$ is a free $\cRing$-bimodule, where $\cRing=\kk[x_{j}^{\pm}]_{j\in I_{\fv}}$.
Similarly, $\LP'$ is a free $\cRing'$-bimodule. Using the restriction
map $\var:\cR\rightarrow\cR'$, we endow $\LP'$ with a free $\cRing$-bimodule
structure. Then we have the following $\cRing'$-module isomorphism
\begin{align*}
\varphi: & \cRing'\otimes_{\cRing}\LP\simeq\LP'
\end{align*}
such that $\varphi(p'\otimes x^{m}):=p'*\var(x^{m})$, for any frozen
factor $p'\in\cRing'$ and $m\in\Mcirc$.

\begin{defn}[{\cite[Definition 4.13]{qin2023analogs}}]

The map $\varphi$ is called the base change from $\LP$ to $\LP'$
associated with $\var$.

\end{defn}

\begin{thm}[{\cite[Theorems 4.14, 4.15]{qin2023analogs}}]\label{thm:base-change-upclalg}

(1) $\varphi$ restricts to a $\kk$-algebra isomorphism $\cRing'\otimes_{\cRing}\clAlg\simeq\clAlg'$.

(2) If $\upClAlg$ has a $\kk$-basis satisfying condition (C1)(C2),
or if $\upClAlg'$ has this property, then $\varphi$ restricts to
a $\kk$-algebra isomorphism $\cRing'\otimes_{\cRing}\upClAlg\simeq\upClAlg'$.

\end{thm}

The map $\varphi$ is called a base change for the following reason.
Denote $\alg=\clAlg$ or $\upClAlg$. When $\kk=\Z$ and the condition
in Theorem \ref{thm:base-change-upclalg}(2) is satisfied, we have
a base change of scheme 
\begin{align*}
\varphi^{*} & :\spec\alg'\simeq\spec\cRing'\times_{\spec\cRing}\alg.
\end{align*}

In general, we do not know if there exists a variation map $\var$
from $\LP$ to $\LP'$. However, it is always possible to choose a
variation map $\var_{1}:\LP(\sd^{\prin},\lambda_{1})\rightarrow\LP(\sd)$
and a variation map $\var_{2}:\LP(\sd^{\prin},\lambda_{2})\rightarrow\LP(\sd')$.
Here, $\sd^{\prin}$ is the classical seed of principal coefficient
similar to $\sd$ (and thus similar to $\sd'$). For $i=1,2$, $(\sd^{\prin},\lambda_{i})$
is the quantum seed whose Poisson form is $\lambda_{i}$. $\var_{1}$
(resp. $\var_{2}$) is the homomorphism from $\LP(\sd^{\prin},\lambda_{1})$
to $\LP(\sd)$ (resp. to $\LP(\sd')$) as in Example \ref{eg:principal-coeff-seed}. 

In many situations, we can easily propagate structures and properties
between the cluster algebras $\alg(\sd^{\prin},\lambda_{1})$ and
$\alg(\sd^{\prin},\lambda_{2})$, because they only differ by a choice
quantization. Furthermore, they are related to $\alg'$ and $\alg'$
by the base change $\varphi_{i}$ associated with $\var_{i}$, $i=1,2$,
respectively. And we have seen that base change also propagates interesting
structures and properties. See the commutative diagram \eqref{eg:base-change-full}
for a summery.
\begin{align}
\begin{array}{ccccccc}
\alg & \xleftarrow{\varphi_{1}} & \alg(\sd^{\prin},\lambda_{1}) & \overset{\text{quantization}}{\underset{\text{change}}{\sim}} & \alg(\sd^{\prin},\lambda_{2}) & \xrightarrow{\varphi_{2}} & \alg'\\
\uparrow &  & \uparrow &  & \uparrow &  & \uparrow\\
\cRing & \xleftarrow{\var_{1}} & \cRing(\sd^{\prin},\lambda_{1}) & \overset{\text{quantization}}{\underset{\text{change}}{\sim}} & \cRing(\sd^{\prin},\lambda_{2}) & \xrightarrow{\var_{2}} & \cRing'
\end{array}\label{eg:base-change-full}
\end{align}

\begin{eg}\label{eg:unipotent-subgroup}

Choose $C=\left(\begin{array}{cccc}
2 & -1 & 0 & 0\\
-1 & 2 & -1 & 0\\
0 & -1 & 2 & -1\\
0 & 0 & -1 & 2
\end{array}\right)$ and $\ubj=(\bj_{k})_{k\in[1,6]}=(1,2,1,3,2,1)$. Then $\sd'=\rsd(\ubj)$
is similar to $\frz_{\{4\}}\sd$ in Example \ref{eg:freezing}. Its
ice quiver is depicted in Figure \ref{fig:A3-GHL}. 

\begin{figure}
\caption{The quiver for a seed of $\kk[N]$, $N\subset SL_{4}$ \cite{GeissLeclercSchroeer11}}
\label{fig:A3-GHL}

\begin{tikzpicture}
 [scale=1.5,node distance=48pt,on grid,>={Stealth[length=4pt,round]},bend angle=45, inner sep=0pt]

\node[unfrozen] (v1) at (0.5,-0.5) {1};
 \node[unfrozen] (v2) at (0pt,-28.45pt) {2};
 \node[unfrozen] (v3) at (-14.225pt,-14.225pt) {3};
 \node[frozen] (v4) at (-14.225pt,-42.675pt) {4};
 \node[frozen] (v5) at (-28.45pt,-28.45pt) {5};
 \node[frozen] (v6) at (-42.675pt,-14.225pt) {6};

\draw[->] (v2) -- (v1);
\draw[->] (v4) -- (v2);
  \draw[->] (v5) -- (v3); 
\draw[->] (v3) -- (v2);
\draw[->,dashed] (v6) -- (v5);
\draw[->,dashed] (v5) -- (v4); 

 \draw[->] (v1) -- (v3); 
  \draw[->] (v3) -- (v6); 
    \draw[->] (v2) -- (v5);

\end{tikzpicture}
\end{figure}

The cluster algebra $\bUpClAlg(\frz_{\{4\}}\sd)$ possesses a common
triangular basis element $\frz_{\{4\}}z=x_{2}^{-1}\cdot x_{6}\cdot(1+y_{2}+y_{1}\cdot y_{2})$
(see Section \ref{sec:Common-triangular-bases}). Take the similar
element $z'=x_{2}^{-1}\cdot x_{4}\cdot(1+y_{2}+y_{1}\cdot y_{2})$
in $\LP(\sd')$. Then $z'$ is the $(-f_{2}+f_{4})$-pointed common
triangular basis element of $\bUpClAlg(\sd')$. Under the bijection
$\kappa:\bUpClAlg(\sd')\simeq\kk[N]$, $z'$ is identified with a
dual canonical basis element of $\kk[N]$ up to a scalar multiple
of $v$ (\cite{qin2017triangular}), where $N$ is the subgroup consisting
of unipotent upper triangular matrices of $SL_{3}$ (Section \ref{subsec:Cluster-structures-on-G}).

\end{eg}

\section{A review of Kazhdan-Lusztig algorithms}\label{sec:A-review-Kazhdan-Lusztig}

\subsection{Kazhdan-Lusztig algorithm}\label{subsec:Kazhdan-Lusztig-algorithm}

Let $E$ denote a free $\Z[v^{\pm}]$-module, endowed with a $\Z$-linear
involution $\overline{(\ )}$ sending $v$ to $v^{-1}$, called the
bar involution. We make the following assumptions:
\begin{itemize}
\item $E$ has a $\Z[v^{\pm}]$-basis $\std=\{\std_{w}|w\in W\}$, where
$W$ denotes an index set. 
\item $W$ has an partial order $<$ which is bounded from below. It induces
a natural order on $\std$.
\item For any $\std_{w}$, we have $\overline{\std_{w}}-\std_{w}\in\sum_{w'<w}b_{w',w}\std_{w'}$,
where $b_{w',w}\in\Z[v^{\pm}]$.
\end{itemize}
\begin{lem}[{\cite[Lemma 8.4]{Nakajima04}\cite[7.10]{Lusztig90}}]\label{lem:KL-basis}

Choose $\mm=v\Z[v]$ or $v^{-1}\Z[v^{-1}]$. The module $E$ has a
unique $\Z[v^{\pm}]$-basis $\base=\{\base_{w}|w\in W\}$, such that
\begin{itemize}
\item $\overline{\base_{w}}=\base_{w}$.
\item $\base_{w}=\std_{w}+\sum_{w'<w}c_{w',w}\std_{w'}$, where $c_{w',w}\in\mm$.
\end{itemize}
\end{lem}

\begin{proof}

Since $<$ is bounded from below on $W$, the basis $\base$ could
be recursively determined by $\std$. 

More precisely, for any given $w$, assume that $\base_{w'}$ have
been determined for all $w'<w$. We decompose $\overline{\std_{w}}-\std_{w}=\sum_{w'<w}d_{w',w}\base_{w'}$
where $d_{w',w}\in\Z[v^{\pm}]$. Note that $\overline{d_{w',w}}=-d_{w',w}$.
Uniquely decompose $d_{w',w}=\overline{e_{w',w}}-e_{w',w}$ such that
$e_{w',w}\in\mm$. Then $\base_{w}:=\std_{w}-\sum_{w'<w}e_{w',w}\base_{w'}$
is bar-invariant. In this way, we recursively construct the desired
basis $\base$ satisfying all the conditions.

To show the uniqueness, assume that $\base'=\{\base'_{w}|w\in W\}$
is another such basis. If $w$ is minimal in $W$, we have $\base'_{w}=\std_{w}=\base_{w}$.
Inductively, assume we have shown $\base_{w'}=\base'_{w'}$ for all
$w'<w$. Denote $\std_{w}=\base'_{w}+\sum e'_{w',w}\base'_{w'}$,
where $e'_{w',w}\in\mm$. Then we have $0=\std_{w}-\std_{w}=(\base_{w}-\base'_{w})+\sum_{w'<w}(e_{w',w}-e'_{w',w})\base_{w'}$.
The bar-invariance implies that $e_{w',w}-e'_{w',w}=0$ and, consequently,
$\base_{w}=\base'_{w}$.

\end{proof}

When $E$ is the Hecke algebra of the Coxeter group $W$, we can choose
$\std$ to be its standard basis $\{T_{w}|w\in W\}$, $<$ to be the
Bruhat order, and $\mm=v\Z[v]$. In this case, the resulting basis
$\base$ is called its canonical basis, and the transition coefficients
$c_{w',w}$ for $\base_{w}=\std_{w}+\sum_{w'<w}c_{w,w'}\std_{w'}$
are called the Kazhdan-Lusztig polynomials. 

From now on, we call the recursive algorithm for constructing $\base$
in the proof of Lemma \ref{lem:KL-basis} the Kazhdan-Lusztig algorithm
(the KL algorithm for short). The basis $\base$ will be called the
Kazhdan-Lusztig-type basis (the KL basis for short). 

\begin{rem}\label{rem:KL-algorithm-correction}

It may be helpful to think of the KL algorithm as sharing, in some
sense, the spirit of correction as the classical Gram--Schmidt procedure:
\begin{itemize}
\item Let there be given an initial basis $\std$ which is sorted by some
(partial) order $<$. 
\item Modifying each the basis elements by adding correction terms (basis
elements inferior in $<$), we obtain a new basis $\base$ with desired
properties.
\end{itemize}
\end{rem}

\begin{rem}

The above algorithm works for $\Q(v)$-modules as well. For example,
let $W$ be a finite Weyl gorup and $l(w_{0})$ be the length of its
longest element. When $E$ is the quantized enveloping algebra $U_{q}(\frn)$,
we can choose $\std$ to be its PBW basis, endow $W=\N^{l(w_{0})}$
with the lexicographical order, and take $\mm=v^{-1}\Z[v^{-1}]$.
Then the resulting basis $\base$ is the canonical basis.

\end{rem}

\subsection{Categorification}\label{subsec:Categorification}

Consider the base change in the KL algorithm:
\begin{align}
\std_{w} & =\base_{w}+\sum_{w'<w}e_{w',w}\base_{w'}\label{eq:std-to-KL}
\end{align}
In literature, it often happens that the coefficients $e_{w',w}$
take values in $\mm_{+}:=v\N[v]$ or $v^{-1}\N[v^{-1}]$. Moreover,
$\base$ often has positive structures constants: $\forall w,w'\in W$,
we have
\begin{align*}
\base_{w}*\base_{w'} & =\sum_{w''\in W}b_{w''}\base_{w''},\quad b_{w''}\in\N[v^{\pm}].
\end{align*}
These properties are far from obvious from the KL algorithm. They
are often implied by categorification, which we will explain below.

Let $\F$ denote a field. Let $(\cT,\otimes)$ denote a $\F$-linear
category endowed with an exact bifunctort $\otimes$ called the tensor
functor. We assume $\cT$ is a tensor category in the sense of \cite[Section A.1]{kang2018symmetric},
which we will also call a monoidal category. We endow its Grothendieck
ring $K_{0}(\cM)$ with the multiplication $*$ given by the tensor
product $\otimes$:
\begin{align*}
[X]*[Y] & :=[X\otimes Y]
\end{align*}
where $X,Y\in\cT$ and $[X],[Y]$ denote their isoclassess in $K_{0}(\cT)$.
Under our definition of $\cT$, $K_{0}(\cT)$ is associative, unital,
and has a $\Z$-basis $\simp=\{[S]|S\text{ is a simple object}\}$.
Moreover, there is an object $Q$ in $\cT$ such that $\bq:=Q\otimes(\ )$
is an autoequivalence functor on $\cT$. This makes $K_{0}(\cT)$
a $\Z[q^{\pm}]$-algebra where $q[X]:=[\bq X]$.

Denote $\kk=\Z[v^{\pm}]$, $\Q[v^{\pm}]$, $\Q(v)$. Take the $\kk$-module
$K=\kk\otimes_{\Z[q^{\pm}]}K_{0}(\cT)$ such that $v=q^{\Hf}$. If
$K_{0}$ is non-commutative, it will often induces a multiplication
for $K$. Otherwise, we could often endow $K$ with a non-commutative
multiplication $*$ (with desired properties), such that the classical
limit $K|_{v\mapsto1}$ is $K_{0}$. The algebra $K$ is called the
quantum Grothendieck ring of $\cT$. 

Let $\cT^{\op}$ denote the opposite monoidal category of $\cT$ in
the sense of \cite[Definition 2.1.5]{etingof2015tensor}. In particular,
its tensor product $\otimes^{\op}$ is given by $X\otimes^{\op}Y:=Y\otimes X$
for $X,Y\in\cT$. Define its quantum Grothendieck ring to be the opposite
algebra $K^{\op}$.

Let $R$ be an $\kk$-algebra with a KL basis $\base$. 

\begin{defn}\label{defn:categorification}

We say a monoidal category $\cT$ categorifies the pair $(R,\base)$
is categorified if the following conditions are satisfied.
\begin{itemize}
\item There is a $\kk$-algebra isomorphism $\kappa:R\simeq K$.
\item For any simple object $S$ of $\cT$, we have $v^{\alpha}[S]\in\kappa\base$
for some $\alpha\in\Z$.
\end{itemize}
\end{defn}

Assume $\cT$ categorifies $(R,\base)$. If $\std_{w}$ also corresponds
to an object in $\cT$, then coefficients $e_{w',w}$ are the Jordan--Hölder
multiplicities of its composition factors. Thus we must have $e_{w',w}\in\mm_{+}$.

In analogous to based modules in \cite{Lus:intro}, we might call
the pair $(R,\base)$ a based algebra, if it has some additional desired
properties.

\begin{rem}

We have extended the Grothendieck ring $K_{0}(\cT)$ to a $\kk$-algebra
$K$ for applications in quantum cluster algebras; see \cite[Section 5.1]{qin2023analogs}.
This technical treatment might not be necessary for other purposes.

Sometimes, the construction of a suitable $\kk$-algebra $K$ from
$\cT$ might be highly non-trivial, e.g., one might want to deform
a commutative Grothendieck ring $K_{0}(\cT)$ to a non-commutative
ring $K$ when studying representations of quantum affine algebras
\cite{Nakajima01} \cite{VaragnoloVasserot03} \cite{Hernandez02}
(Section \ref{subsec:Cluster-algebras-from-q-aff-alg}).

\end{rem}

\section{Common triangular bases }\label{sec:Common-triangular-bases}

Let $\alg$ denote an quantum ordinary cluster algebra $\clAlg$ or
an quantum upper cluster algebra $\upClAlg$. In general, there is
no known initial basis $\std$ for $\alg$ on which we could apply
the KL algorithm (Section \ref{sec:A-review-Kazhdan-Lusztig}). Following
\cite{qin2017triangular}, we hope to construct its KL basis in the
following procedure:
\begin{enumerate}
\item We choose any seed $\sd$.
\item We will construct a distinguished linearly independent set $\Inj^{\sd}$
in $\LP(\sd)$ instead of an initial basis $\std$. For our construction
in this step, \textbf{we need to assume that the set of vertices $I$
is finite}.
\item We sort the elements of $\Inj^{\sd}$ by an unbounded partial order
$\prec_{\sd}$.
\item We add correction terms to $\Inj^{\sd}$ using the KL algorithm (Remark
\ref{rem:KL-algorithm-correction}), which might not stop since $\prec_{\sd}$
is unbounded. We will obtain a linearly independent set $\can^{\sd}$
in the completion $\hLP(\sd)$.
\end{enumerate}
It might happen that $\can^{\sd}$ is actually a basis for $\alg$.
Then it will be called the triangular basis for $\alg$ with respect
to $\sd$ if it satisfies some additional technical condition (Definition
\ref{defn:tri-basis}). If $\alg$ has a basis $\can$ such that $\can=\can^{\sd}$
for any seed $\sd$, it will be call the common triangular basis (Definition
\ref{def:common-tri-basis}).

The KL-algorithm implies that the triangular basis $\can^{\sd}$ is
unique if it exists. The uniqueness of $\can$ follows.

\begin{rem}\label{rem:BZ-triangular-basis}

In \cite{BerensteinZelevinsky2012}, Berenstein and Zelevinsky constructed
a KL basis for the quantum cluster algebra $\clAlg(\sd)$ where $\sd$
is any acyclic seed. By acyclic, we mean that there is an order $<$
on $I_{\ufv}$ such that the exchange matrix $B$ satisfies $b_{ij}=0$
whenever $i\leq j$. Using this order, they constructed a basis $\std$
consisting of products of cluster variables in $\sd$ and $\mu_{k}\sd$,
where $k\in I_{\ufv}$. Sorting the elements of $\std$ lexicographically
and applying the KL algorithm, they obtained a basis $\base$, which
we call the BZ-triangular basis $\can^{\BZ}$. They further showed
that different acyclic seeds produce the same basis.

By \cite{qin2019compare,qin2020dual}, $\can^{\BZ}$ coincides with
the common triangular basis $\can$ for the acyclic quantum cluster
algebra.

\end{rem}

\subsection{Triangular bases}

Let us explain how to construct bases in details. Choose any seed
$\sd$, which is assumed to be injective reachable. Note that $|I|$
must be finite under this assumption.

Let $\Inj^{\sd}$ denote the set of all normalized products of the
form $[p*x(\sd)^{g_{+}}*x(\sd[1])^{g_{-}}]^{\sd}$, where $p$ is
a frozen factor, $g_{\pm}\in\N^{I_{\ufv}}$, and $(g_{+})_{k}(g_{-})_{k}=0$
$\forall k\in I_{\ufv}$ (i.e., the product are reduced). Then $\Inj^{\sd}$
is a $\Mcirc(\sd)$-pointed set. In particular, it is linearly independent.
Let $\Inj_{m}^{\sd}$ denote the $m$-pointed element in $\Inj^{\sd}$
as before.

Recall that we have the dominance order $\prec_{\sd}$ on $M(\sd)$.
Choose $\mm=v^{-1}\Z[v^{-1}]$. In analogous to Lemma \ref{lem:KL-basis},
we have the following result.

\begin{lem}\label{lem:inj-tri}

There is a unique subset $\can^{\sd}=\{\can_{m}^{\sd}|m\in M(\sd)\}$
of $\hLP(\sd)$ such that
\begin{itemize}
\item $\overline{\can_{m}^{\sd}}=\can_{m}^{\sd}$.
\item $\can_{m}^{\sd}=\Inj_{m}^{\sd}+\sum_{m'\prec_{\sd}m}c_{m',m}\Inj_{m'}^{\sd}$
in $\hLP(\sd)$, where $c_{m',m}\in\mm$.
\end{itemize}
\end{lem}

The decomposition in the second condition in Lemma \ref{lem:inj-tri}
will be called a $(\prec_{\sd},\mm)$-decomposition. Note that it
could be an infinite sum. Under an appropriate topology of $\bLP(\sd)$,
$\Inj$ becomes a topological basis of $\bLP(\sd)$, and this infinite
sum becomes convergent; see \cite{davison2019strong}.

\begin{proof}

The basis could be constructed via the KL algorithm, which should
be viewed as an infinite correction process (Remark \ref{rem:KL-algorithm-correction}).
See \cite[Theorem 6.1.3, Lemma 6.1.4]{qin2020dual} or \cite[7.10]{Lusztig90}
for a detailed treatment.

\end{proof}

Let $\alg$ denotes the quantum cluster algebra $\clAlg$ or $\upClAlg$.

\begin{defn}[{\cite[Definition 6.1.1]{qin2017triangular}}]\label{defn:tri-basis}

Assume $\sd$ is finite-rank and injective reachable. Let $\can^{\sd}=\{\can_{m}^{\sd}|m\in\Mcirc(\sd)\}$
denote a $\kk$-basis of $\alg$, such that $\can_{m}^{\sd}$ are
$m$-pointed. It is called the triangular basis $\can^{\sd}$ of $\alg$
with respect to $\sd$, if it satisfies the following conditions.
\begin{itemize}
\item Bar-invariance: $\overline{\can_{m}^{\sd}}=\can_{m}^{\sd}$.
\item $\can^{\sd}$ contains all cluster monomials of $\sd$ and $\sd[1]$.
\item Triangularity: $\forall m\in\Mcirc(\sd)$ and $i\in I$, we have the
following $(\prec_{\sd},\mm)$-triangular decomposition
\begin{align*}
[x_{i}(\sd)*\can_{m}^{\sd}]^{\sd} & =\can_{m+f_{i}}^{\sd}+\sum_{m'\prec_{\sd}m}b_{m'}\can_{m'}^{\sd},\ b_{m'}\in\mm.
\end{align*}
\end{itemize}
\end{defn}

Let $\can^{\sd}$ denote the triangular basis with respect to $\sd$.
By definition, it must be the unique subset in Lemma \ref{lem:inj-tri}.

\subsection{Common triangular bases}

\begin{defn}[{Common triangular basis \cite{qin2017triangular}\cite{qin2021cluster}}]\label{def:common-tri-basis}

If $\can$ is a $\kk$-basis of $\alg$ such that $\can$ is the triangular
basis $\can^{\sd}$ for any seed $\sd$, it will be call the common
triangular basis of $\alg$. If further, $\can\cap\overline{\alg}$
is a basis of $\overline{\alg}$, $\can\cap\overline{\alg}$ is called
the common triangular basis of $\overline{\alg}$.

\end{defn}

Assume $\alg$ has the common triangular basis $\can$. Then it contains
all cluster monomials by definition. \cite[Proposition 6.4.3]{qin2020dual}
implies the following. 

\begin{prop}

The common triangular basis $\can$ is $\tropM$-pointed.

\end{prop}

\begin{rem}

By \cite[Theorem 4.3.1]{qin2019bases}, when $\sd$ is injective reachable,
any $\tropM$-pointed subset of $\upClAlg$ is a basis for $\upClAlg$.
We deduce that $\alg=\upClAlg$ in this case.

Therefore, for studying common triangular bases, it is sufficient
to consider the upper cluster algebra $\alg=\upClAlg$. In practice,
we might also start with the ordinary cluster algebra $\alg=\clAlg$,
since $\clAlg$ naturally appears and is well understood in some situations.
On the contrary, it might not worth the effort to study the triangular
basis for a general subalgebra of $\upClAlg$.

\end{rem}

\begin{thm}[{\cite{qin2017triangular}\cite{qin2020dual}}]\label{thm:tri-basis-is-canonical}

Let $C$ be any generalized Cartan matrix.

(1) Take any Weyl group element $w\in W$. Endow the quantized coordinate
ring $\kk[N(w)]$ of a unipotent subgroup $N(w)$ with the canonical
cluster structure (\cite{Kimura10} \cite{GeissLeclercSchroeer11}
\cite{GY13,goodearl2020integral}). Then the dual canonical basis
of $\kk[N(w)]$ is its common triangular basis up to $v$-multiples.

(2) Assume $C$ is of type $ADE$. Let $\cO_{l}$ denote a level-$l$
category consisting of representations of the quantum affine algebra
$U_{q}(\hat{\frg})$ in the sense of \cite{HernandezLeclerc09} (Section
\ref{subsec:Cluster-algebras-from-q-aff-alg}). Endow the quantized
Grothendieck ring $K_{t}(\cO_{l})$ with the canonical cluster structure
(\cite{HernandezLeclerc09}\cite{qin2017triangular}). Then the set
of isoclasses of simple objects is its common triangular basis.

\end{thm}

\begin{rem}\label{rem:tri-basis-unip-cell}

Take any reduced word $\ueta$ such that $w_{\ueta}=w$. Then we have
a canonical bijection $\kappa:\bUpClAlg(\rsd(\ueta))\simeq\kk[N(w)]$
(see \cite{GeissLeclercSchroeer11}\cite{GY13,goodearl2020integral}
or \cite[Section 8]{qin2020dual}). We also have the De Concini-Procesi
isomorphism $\kk[N(w)][(\kappa x_{j})^{-1}]_{j\in I_{\fv}}\simeq\kk[N^{w}]$,
where $\kk[N^{w}]$ is the quantized coordinate ring of the unipotent
cell \cite{kimura2017twist}. Thus we obtain a cluster structure $\upClAlg(\rsd(\ueta))\simeq\kk[N^{w}]$.
Up to $v$-multiples, the common triangular basis of cluster algebra
$\upClAlg(\rsd(\ueta))$ is obtained from the dual canonical basis
after localization.

In practice, it is convenient to work with either the cluster algebra
$\bUpClAlg(\rsd(\ueta))$, which has the standard basis (Section \ref{subsec:Standard-bases},
dual PBW basis), or with $\upClAlg(\rsd(\ueta))$, to which we could
apply cluster operations (Section \ref{sec:Cluster-operations}).

\end{rem}

\subsection{Existence of the common triangular bases}

The common triangular basis is defined to have strong properties at
the cost that its existence is difficult to prove. Note that $\sd[1]=\Sigma\sd$
and $\sd=(\sigma^{-1}\Sigma)\sd[-1]$ (or equivalently, $\sd[-1]=(\sigma^{-1}\Sigma)^{-1}\sd$).
We have the following necessary and sufficient condition for its existence.

\begin{thm}[{A slightly different version of \cite[Theorem 6.5.4]{qin2020dual}}]\label{thm:tri-basis-existence-criterion}

If $\can$ is the triangular basis $\can^{\sd}$ for $\alg$ with
respect to the seed $\sd$, and it contains the cluster monomials
of the seeds appearing along the mutation sequences $\Sigma$ and
$(\sigma^{-1}\Sigma)^{-1}$starting from $\sd$. Then $\can$ is the
common triangular basis for $\alg$.

\end{thm}

Cluster operations propagate the existence of (common) triangular
basis. Assume $\sd'$ is a seed similar to $\sd$, and $\frz_{F}\sd$
is obtained from $\sd$ by freezing $F\subset I_{\ufv}$. Let $\alg$
denote $\clAlg$ or $\upClAlg$. Assume $\can$ is the triangular
basis for $\alg(\sd)$ with respect to $\sd$. Let $\can'$ denote
the set of similar elements in $\LP(\sd')$ and $\frz_{F}\can$ denote
the set obtained from $\can$ by freezing.

\begin{thm}[{\cite{qin2023analogs}}]

The following statements are true.

(1) $\can'$ is the triangular basis for $\alg(\sd')$ with respect
to $\sd'$. 

(2) If $\can$ is the common triangular basis, then $\can'$ is the
common triangular basis too.

(3) $\frz_{F}\can$ is the triangular basis for $\alg(\frz_{F}\sd)$
with respect to $\frz_{F}\sd$. 

(4) If $\can$ is the common triangular basis, then $\frz_{F}\can$
is the common triangular basis too.

\end{thm}

\begin{proof}

The desired the statements could be deduced from a combination of
the correction technique (Theorem \ref{thm:correction-technique}),
similarity of the localized cluster monomials (Theorem \ref{thm:similar-cluster-var}),
and freezing of cluster monomials (Theorem \ref{thm:freeze-cluster-monomial}).

\end{proof}

\subsection{Common triangular bases for known cluster algebras}\label{subsec:Common-triangular-bases-known}

There is an ever growing list of cluster algebras that arise from
Lie theory. More precisely, we could construct (classical and quantum)
upper cluster algebra $\upClAlg(\sd)$ whose initial seed $\sd$ arises
from the consideration of the following varieties from Lie theory.
\begin{enumerate}
\item Algebra group $G$: Assume the generalized Cartan matrix is of finite
type. Let $G$ be the associated connected, simply connected, complex
semisimple algebraic group. The coordinate ring $\C[G]$ and the quantized
coordinate ring $\Q(v)[G]$ are upper cluster algebras $\bUpClAlg(\sd)$
defined over $\C$ and $\Q(v)$, respectively; see \cite{qin2025partially}
(the cluster structure on $\C[SL_{n}]$ were shown in \cite{fomin2020introduction};
see \cite{oya2025note} for another approach to $\C[G]$ when $G$
is not of type $F_{4}$).
\item Subvarieties $\AVar$ of $G$: Let $\AVar$ denote unipotent subgroups
$N(w)$, unipotent cells $N^{w}$, or double Bruhat cell $G^{u,v}$,
the (quantized) coordinate rings $\kk[\AVar]$ are know to be cluster
algebras $\bUpClAlg(\sd)$ or $\upClAlg(\sd)$ (see \cite{GeissLeclercSchroeer10,GeissLeclercSchroeer11}\cite{GY13,goodearl2020integral}).
\item Spaces of configurations of flags $\AVar$: Let $\AVar$ denote Grassmannians,
open Positroid varieties, open Richardson varieties, double Bott-Samelson
cells, or braid varieties. Its coordinate ring are know to be classical
upper cluster algebras $\upClAlg(\sd)$; see \cite{scott2006grassmannians}\cite{serhiyenko2019cluster}\cite{galashin2019positroid}\cite{elek2021bott}\cite{shen2021cluster}\cite{galashin2022braid}\cite{casals2022cluster}.
\item Simple Lie group $G$ endowed with non-standard Lie-Poisson brackets
\cite{gekhtman2023unified}
\end{enumerate}
Note that Family (3) includes several (quantum) cluster algebras arising
from monoidal categorifications, in particular, those from representations
of quantum affine algebras \cite{HernandezLeclerc09}\cite{qin2017triangular}
(Section \ref{subsec:Cluster-algebras-from-q-aff-alg}) and in type
$A_{(1)}^{1}$, those from Satake categories \cite{cautis2019cluster}.

\begin{thm}[{\cite{qin2023analogs}}]\label{thm:Lie-tri-basis}

For almost all known quantum upper cluster algebras $\upClAlg$ arising
from Lie theory (families (1)(2)(3)), we have the following results.

(1) $\upClAlg$ possesses the common triangular basis $\can$.

(2) We have $\clAlg=\upClAlg$.

(3) When the generalized Cartan matrix $C$ is symmetric, $\can$
has positive structure constants.

\end{thm}

Family (4) leads to exotic cluster structures on $\C[G]$, which are
not examined in Theorem \ref{thm:Lie-tri-basis}. 

\begin{proof} We refer the reader to \cite{qin2023analogs} for detailed
arguments. We sketch its approach below.

The existence of common triangular bases is already known for those
from quantum unipotent cells $\kk[N^{w}]$ (Theorem \ref{thm:tri-basis-is-canonical}(1),
Remark \ref{rem:tri-basis-unip-cell}). We wish to propagate structures
and results from $\kk[N^{w}]$ to other cluster algebra $\kk[\AVar]$
where $\AVar$ arises from Lie theory. But $\AVar$ might be more
complicated than $N^{w}$, for example, we cannot expect to recover
$G$ from $N^{w}$. To circumvent this difficulty, we introduce and
use the following \textbf{extension and reduction technique}:

\emph{Extend} $C$ \emph{to a larger generalized Cartan matrix $\tC$
by adding rows and columns. Then we might find a quantum unipotent
cell $\kk[\tN^{\tw}]$ associated with $\tC$ such that it is complicated
enough to recover $\kk[\AVar]$ associated with $C$. }

More precisely, we study $\kk[\AVar]$ as below:
\begin{itemize}
\item Choose\emph{ $\kk[\tN^{\tw}]$ associated with an extended generalized
Cartan matrix $\tC$.}
\item Applying cluster operations (freezing and base changes in Section
\ref{sec:Cluster-operations}), we could propagate structures and
results from \emph{$\kk[\tN^{\tw}]$} to $\upClAlg(\rsd)$, where
$\rsd=\rsd(\ubi)$ is a seed associated with a signed word.\footnote{$\rsd(\ubi)$ is a seed for the coordinate ring of a (half decorated)
double Bott-Samelson variety \cite{shen2021cluster}.}
\item Propagate structures and results from $\upClAlg(\rsd)$ to $\kk[\AVar]$. 
\end{itemize}
\end{proof}

It is desirable to find a natural Lie theoretic interpretation of
the approach in the proof of Theorem \ref{thm:Lie-tri-basis}.

\subsection{Quasi-categorifications}\label{subsec:Quasi-categorifications}

On the one hand, there are many cluster algebras, because one can
construct a cluster algebra from any given skew-symmetrizable matrix.
On the other hand, we only know a limited families of categories from
representation theory that possess good properties. In the following,
we relax the conditions on categorification, so that we could use
categories to study more cluster algebras.

Let $\alg$ denote a (classical or quantum) cluster algebra and $\cT$
be a monoidal category. Let $K$ be the (classical or quantum) Grothendieck
ring of $\cT$. Let $[V]$ denote the element in $K$ corresponding
to $V\in\cT$.

\begin{defn}[{\cite[Definition 5.3]{qin2023analogs}}]

We say $\cT$ admits a cluster structure $\alg$ if the following
conditions are satisfied:
\begin{itemize}
\item There is a $\kk$-algebra injective map $\kappa:\alg\hookrightarrow K$.
\item For any cluster monomial $z$, there exists a simple object $S$ of
$\cT$ such that $\kappa z=v^{\alpha}[S]$ for some $\alpha\in\Z$.
\end{itemize}
We further say $\cT$ categorifies $\alg$ if $\kappa$ is an isomorphism.
This is a weaker version of the categorification in \cite{HernandezLeclerc09}.
We say $\cT$ categorifies $\alg$ after localization if $\kappa$
extends to an isomorphism $\kappa:\alg[x_{j}^{-1}]_{j\in I_{\fv}}\simeq K[\kappa(x_{j})^{-1}]_{j\in I_{\fv}}$.

\end{defn}

\begin{defn}[{\cite[Definition 5.1]{qin2023analogs}}]

Let $\base$ be a $\kk$-basis of the cluster algebra $\alg$. The
pair $(\alg,\base)$ is called a based cluster algebra if the following
conditions hold.
\begin{itemize}
\item $\base$ contains all cluster monomials.
\item (Frozen torus action) $\base$ is a closed under commutative multiplication
by frozen variables: $\forall j\in I_{\fv}$, $x_{j}\cdot\base\subset\base$.
\item (Homogeneity) Any $b\in\base$ is contained in some $x^{m}\cdot\kk[N_{\ufv}]$.
\item (Bar invariance) At the quantum level, the elements of $\base$ are
further assumed to be bar-invariant.
\end{itemize}
\end{defn}

\begin{defn}[{\cite[Definition 5.4]{qin2023analogs}}]

Assume $\cT$ is a monoidal category. We say $\cT$ categorifies a
based cluster algebra $(\alg,\base)$ (resp. after localization) if
the following conditions are satisfied:
\begin{itemize}
\item $\cT$ categorifies $\alg$ (resp. after localization).
\item For any simple object $S$ of $\cT$, we have $v^{\alpha}[S]\in\kappa\base'$
for some $\alpha\in\Z$.
\end{itemize}
\end{defn}

\begin{defn}[{\cite[Definition 5.5]{qin2023analogs}}]

We say a monoidal category $\cT$ quasi-categorifies a based cluster
algebra $(\alg,\base)$, if there exists a based cluster algebra $(\alg',\base')$
which is connected to $(\alg,\base)$ by a finite sequence of base
changes and quantization changes, such that $\cT$ categorifies $(\alg,\base)$
after localization.

\end{defn}

Now we have the following categorification result for the bases in
Theorem \ref{thm:Lie-tri-basis}.

\begin{thm}[{\cite{qin2023analogs}}]\label{thm:quasi-categorification-Lie}

For almost all known cluster algebras $\alg$ arising from Lie theory
(families (1)(2)(3)), when the generalized Cartan matrix is symmetric,
the based cluster algebra $(\alg,\can)$ is quasi-categorified by
some non semisimple monoidal category $(\cT,\otimes)$. 

\end{thm}

\begin{proof}

We refer the reader to \cite{qin2023analogs} for detailed treatments.
The approach in loc. cit. was based on the observation that freezing
corresponds to the restriction on a subcategory: if $\cT$ categories
$(\upClAlg,\base)$, $(\frz_{F}\upClAlg,\frz_{F}\base)$ is categorified
by a monoidal subcategory $\cT'$ of $\cT$.

\end{proof}

\subsection{Standard bases}\label{subsec:Standard-bases}

In general, cluster algebras do not have PBW type basis (i.e, a basis
consisting of all monomials of the chosen distinguished elements),
because they are usually not polynomial rings at the classical level.

Nevertheless, let $\bUpClAlg(\rsd)$ denote the (partially) compactified
upper cluster algebra associated with a quantum seed $\rsd=\rsd(\ueta)$
for some word $\ueta=(\eta_{k})_{k\in[1,l]}$. Then it has the standard
basis (PBW type basis), whose construction is given as follows.

For any $w\in\N^{I}$, define the corresponding standard monomial
to be $\std(w)=[W_{1}^{w_{1}}*\cdots*W_{l}^{w_{l}}]^{\rsd}$, where
$W_{k}$ are the fundamental variables (Section \ref{subsec:Injective-reachable-fundamental-var}).
Define $\stdBasis=\{\std(w)|w=(w_{k})_{k\in[1,l]}\in\oplus_{k=1}^{l}\Z\beta_{k}\}$.
We will sort its elements by the lexicographical order: $\std(w)<\std(w')$
if $w<_{\lex}w'$, or by the reverse lexicographical order $\std(w)<\std(w')$
if $w<_{\rev}w'$. Recall that we say $w<_{\lex}w'$ if $w_{[1,r]}=w'_{[1,r]}$
and $w_{r+1}<w'_{r+1}$ for some $r\in[0,l-1]$, and $w<_{\rev}w'$
if $w_{[r,l]}=w'_{[r,l]}$ and $w_{r-1}<w'_{r-1}$ for some $r\in[2,l+1]$.

Following \cite{qin2017triangular}\cite[Section 9.1]{qin2020dual},
define the linear map $\theta^{-1}:\oplus_{k=1}^{l}\Z\beta_{k}\simeq\Mcirc$
such that
\begin{align}
\theta^{-1}(\beta_{k}) & =f_{k}-f_{k[-1]},\ \forall k\in[1,l].\label{eq:wt-to-g-vector}
\end{align}
where we denote $f_{k^{\min}[-1]}=f_{-\infty}=0$ as before. Let $\can_{m}$
denote the $m$-pointed common triangular basis element of $\bUpClAlg(\rsd)$.
We denote $\can(w):=\can_{\theta^{-1}(w)}$.

\begin{thm}[{\cite[Theorem 8.10]{qin2023analogs}}]\label{thm:std-bases}

For any $\rsd=\rsd(\ueta)$, where $\ueta$ is a finite word in $J$,
the following statements are true.
\begin{enumerate}
\item $\stdBasis$ is a $\kk$-basis of $\bUpClAlg(\rsd)$, called the standard
basis. 
\item We have $\bClAlg(\rsd)=\bUpClAlg(\rsd)$.
\item The standard basis satisfies the analog of the Levendorskii-Soibelman
straightening law:
\begin{align*}
W_{k}*W_{j}-v^{\lambda(\deg W_{k},\deg W_{j})}W_{j}*W_{k} & =\sum_{w\in\N^{[j+1,k-1]}}b_{w}\std(w),\ b_{w}\in\kk,\ \forall j\leq k.
\end{align*}
\item The common triangular basis of $\bUpClAlg(\rsd)$ is $\{\can(w)|w\in\oplus_{k=1}^{l}\N\beta_{k}\}$.
Moreover, it is the Kazhdan-Lusztig-type basis associated with the
standard basis $\stdBasis$, whose basis elements are sorted by the
lexicographical order or by the reverse lexicographical order.
\end{enumerate}
\end{thm}

\begin{thm}[{\cite[Theorem 8.10]{qin2023analogs}}]\label{thm:categorify-dBS}

There is a non-semisimple monoidal category $\cT(\beta_{\ueta})$
which categorifies $(\bUpClAlg(\rsd(\ueta)),\can)$. 

\end{thm}

By \cite{qin2023analogs}, we could construct $\cT(\beta_{\ueta})$
as a monoidal subcategory of $\cC_{\Z}$, which is the monoidal category
consisting of finite dimensional modules of quantum affine algebras
$U_{q}(\hat{\frg})$ (Section \ref{subsec:Cluster-algebras-from-q-aff-alg}).
In this construction, the monoidal category $\cT(\beta_{\ueta)})$
depends on the braid element $\beta_{\ueta}$, not on the choice of
the word $\ueta$.

Assume $C$ is of finite type. We could further compute the fundamental
variables as below. First, $W_{\eta_{1}}$ is an initial cluster variable
$x_{1}(\rsd)$. Next, we can embed $\bUpClAlg(\rsd)=\bClAlg(\rsd)$
as a subalgebra of an infinite rank cluster algebra $\bClAlg_{\infty}$
(\cite{qin2024infinite}, Section \ref{sec:Extension-to-infinite}).
The cluster algebra $\bClAlg_{\infty}$ is isomorphic to the virtual
quantum Grothendieck ring $K_{\infty}$ of a quantum affine algebra
(also called the Bosonic extension of the quantized enveloping algebra);
see \cite{jang2023quantization}. Moreover, the braid group $\Br$
acts on $K_{\infty}$ by \cite{jang2023braid}\cite{kashiwara2024braid}.
Then the remaining $W_{k}$ could be computed by using the action
of $\Br$ as follows.

\begin{thm}[{\cite[Theorem 1.4]{qin2024infinite}}]

Assume $C$ is of finite type, then the fundamental variables of $\bUpClAlg(\rsd(\ueta))$
could be computed by using the braid group action: 
\begin{align*}
W_{k} & =\sigma_{\eta_{1}}\cdots\sigma_{\eta_{k-1}}W_{\eta_{k}},\ \forall k\in[2,l].
\end{align*}

\end{thm}

\section{Extension to infinite ranks}\label{sec:Extension-to-infinite}

\subsection{Colimit of seeds}

\begin{defn}[{\cite{qin2024infinite}}]

We say $\sd$ is a good subseed of $\sd'$, if we have $I\subset I'$,
$I_{\ufv}\subset I'_{\ufv}$, $\sym_{i}=\sym'_{i}$ $\forall i\in I$,
$\tB=(\tB')_{I\times I_{\ufv}}$, and $b'_{jk}=0$ $\forall(j,k)\in(I'\backslash I)\times I_{\ufv}$. 

If $\sd$ and $\sd'$ are quantum seeds, we further require $\Lambda=(\Lambda')_{I\times I}$.

\end{defn}

Now, let $\sd$ be a good subseed of $\sd'$. View $\Mcirc$ as a
sublattice of $\Mcirc'$ naturally. We obtain natural inclusions $\LP\subset\LP'$
and $\cF\subset\cF'$. We further have the following results.

\begin{prop}[{\cite[Section 3.1]{qin2023analogs}}]

We have $\bClAlg\subset\bClAlg'$ and $\bLP(\seq\sd)\subset\bLP(\seq\sd')$
for any mutation sequence $\seq$ on $I_{\ufv}$. Moreover, if $|I'|<\infty$
and $\sd'$ satisfies the injectivity assumption, we have $\bUpClAlg\subset\bUpClAlg'$.

\end{prop}

Let there be a chain of quantum or classical seeds $(\sd_{r})_{r\in\N}$
such that $\sd_{r}$ is a good subseed of $\sd_{r+1}$. 

First assume $\sd_{r}$ are classical seed. Then we could always construct
their colimit seed $\sd_{\infty}$ such that $I(\sd_{\infty}):=I_{\infty}:=\cup_{r}I(\sd_{r})$,
$I_{\ufv}(\sd_{\infty}):=\cup_{r}I_{\ufv}(\sd_{r})$, $\sym_{i}=\sym_{i}(\sd_{r})$
when $i\in I(\sd_{r})$, and $\tB=(b_{ij})_{i\in I_{\infty},j\in I_{\ufv}(\sd_{\infty})}$
such that $b_{ij}=\tB(\sd_{r})_{ij}$ when $(i,j)\in I(\sd_{r})\times I_{\ufv}(\sd_{r})$. 

Next assume $\sd_{r}$ are quantum seed. If we can equip the classical
seed $\sd_{\infty}$ with a compatible Poisson structure $\lambda(\sd_{\infty})$,
such that $\Lambda(\sd_{r})=\Lambda(\sd_{\infty})_{I(\sd_{r})\times I(\sd_{r})}$
for any $r$, we call the corresponding quantum seed the colimit $\sd_{\infty}$. 

\subsection{Extension of quantizations}

If $\sd$ is a good subseed of $\sd'$ as classical seeds, it is nontrivial
to extend a compatible Poisson structure on $\sd$ to that of $\sd'$
such that $\sd$ becomes a good subseed of $\sd'$ as a quantum seed.
Nevertheless, we could restrict on special cases that suffice for
many applications.

\begin{defn}

A matrix $(b_{ij})$ is said to be connected if, $\forall(i,j)$,
there exists a sequence of letters $i_{0}=i,i_{1},\ldots,i_{s-1},i_{s}=j$,
such that $b_{i_{r}i_{r+1}}\neq0$, $\forall r\in[0,s-1]$.

\end{defn}

Assume that we have a partition $I'=I_{1}\sqcup I_{2}\sqcup I_{3}$
such that $I'_{\ufv}=I=I_{1}\sqcup I_{2}$, $I'_{\fv}=I_{3}$, $I_{\ufv}=I_{1}$,
$I_{\fv}=I_{2}$. Then we have $\tB_{I_{3},I_{1}}=0$ since $\sd$
is a good subseed of $\sd'$.

\begin{lem}\label{lem:extend-quantization}

Assume that $\tB_{I'_{\ufv},I'_{\ufv}}$ and $\tB_{I_{\ufv},I_{\ufv}}$
are connected, $\tB_{I_{3},I_{2}}$ is of full rank. Then we can uniquely
extends the compatible Poisson structure $\lambda$ for $\sd$ to
a compatible Poisson structure $\lambda'$ for $\sd'$, such that
$\Lambda=\Lambda'_{I\times I}$.

\end{lem}

\begin{eg}\label{eg:uniqe-A1-quantization}

We can solve the compatible Poisson form $\lambda$ in Example \ref{eg:GHL-A1-inf-quantize}
such that $\diag_{k}=2$. Let $\lambda$ be such a Poisson form and
$\Lambda$ its matrix. Then, for any $0<r\in\N$, we have $\Lambda_{[-r,r]\times[-r,r]}B_{[-r,r]\times[-r+1,r-1]}=\left(\begin{array}{c}
0\\
-2\Id_{2r-1}\\
0
\end{array}\right)$. In particular, when $r=1$, we can obtain a one parameter family
of solutions $\Lambda_{[-1,1]}=\left(\begin{array}{ccc}
0 & 1+c & 0\\
-1-c & 0 & -1+c\\
0 & 1-c & 0
\end{array}\right)$, $\forall c\in\Z$. Using Lemma \ref{lem:extend-quantization} and
increasing $r$ by $1$ each time, we could inductively determine
all $\Lambda_{[-r,r]}$ step by step for all $r>0$. The resulting
matrix $\Lambda$ is given by \eqref{eq:GHL-A1-inf-Lambda} when we
choose $c=0$.

\end{eg}

\subsection{Colimit of algebras}

Now assume $(\sd_{r})_{r\in\N}$ is a chain of quantum or classical
seeds, such that $\sd_{r}$ is a good subseed of $\sd_{r+1}$, and
it has the colimit seed $\sd_{\infty}$. Note that $\LP(\sd_{\infty})=\cup_{r}\LP(\sd_{r})$.
We define the colimit $\bClAlg_{\infty}:=\cup_{r}\bClAlg(\sd_{r})$
and $\bUpClAlg_{\infty}:=\cup_{r}\bUpClAlg(\sd_{r})$

\begin{lem}[{\cite{qin2024infinite}}]\label{lem:extension-colimit}

(1) We have $\bClAlg(\sd_{\infty}):=\bClAlg_{\infty}$.

(2) Assume that $|I(\sd_{r})|<\infty$ and $\sd_{r}$ satisfies the
Injectivity Assumption, $\forall r$. Then we have $\bUpClAlg(\sd_{\infty})=\bUpClAlg_{\infty}$. 

\end{lem}

\begin{thm}[Starfish Lemma]\label{thm:starfish-inf}

Under the assumption in Lemma \ref{lem:extension-colimit}(2), we
have 
\begin{align*}
\bUpClAlg(\sd_{\infty}) & =\bLP(\sd_{\infty})\cap\left(\cap_{k\in I_{\ufv}(\sd_{\infty})}\bLP(\mu_{k}\sd_{\infty})\right).
\end{align*}

\end{thm}

Theorem \ref{thm:starfish-inf} was previously known as the starfish
lemma when $|I|<\infty$ \cite{BerensteinFominZelevinsky05,BerensteinZelevinsky05}.

Now, let $\alg$ denotes $\bClAlg$ or $\bUpClAlg$. Assume $\alg(\sd_{r})$
has the $\kk$-basis $\base_{r}$, such that $\base_{r+1}\cap\alg(\sd_{r})=\base_{r}$.
Denote $\base_{\infty}=\cup_{r}\base_{r}$. Then it is a basis of
$\alg_{\infty}=\cup_{r}\alg(\sd_{r})$. Now we could extend our constructions
of common triangular bases to infinite ranks, despite that $\sd_{\infty}$
is not injective reachable in general.

\begin{defn}

Assume that we have $\alg(\sd_{\infty})=\cup_{r}\alg(\sd_{r})$. If
$\base_{r}$ are the triangular basis of $\alg(\sd_{r})$ with respect
to $\sd_{r}$ (resp. the common triangular basis), then $\base_{\infty}$
is called the triangular basis of $\alg(\sd_{\infty})$ with respect
to $\sd_{\infty}$ (resp. the common triangular basis).

\end{defn}

\begin{eg}\label{eg:A1-quantization-unique}

For $r\in\N$, denote $I_{r}=[-1-r,1+r]$ and $(I_{r})_{\fv}=\{-1-r,1+r\}$.
Denote the (classical or quantum) seed in Examples \ref{eg:GHL-A1-inf}
\ref{eg:GHL-A1-inf-quantize} by $\sd_{\infty}$. Let $\sd_{r}$ denote
the seed obtained from $\sd_{\infty}$ by restricting to the set of
vertices $I_{r}$, where the frozen vertices are in $(I_{r})_{\fv}$.
The associated ice quiver, denoted $\tQ_{r}$, is the full subquiver
on the vertices $[-1-r,1+r]$ of the ice quiver of Figure \ref{fig:A1-GHL},
such that $\{-1-r,1+r\}$ becomes frozen.

Note that we can identify $\sd_{r}$ with the seed $\dsd(\ubi)$ where
$\ubi=(\bi_{k})_{k\in[-r,1+r]}=((1)^{r},1,-1,(1)^{r})$ where $\bi_{1}=-1$.
$(\sd_{r})_{r\in\N}$ is a chain of seeds such that $\sd_{r}$ is
a good subseed of $\sd_{r+1}$, and we have $\bUpClAlg(\sd_{\infty})=\cup_{r\in\N}\bUpClAlg(\sd_{r})$
by Lemma \ref{lem:extension-colimit}. The common triangular basis
$\can_{r}$ of $\bUpClAlg(\sd_{r})$ extends to the common triangular
basis $\can_{\infty}$ of $\bUpClAlg(\sd_{\infty})$.

\end{eg}

\appendix

\section{Examples of Cluster algebras from Lie theory}\label{sec:Examples-of-Cluster}

\subsection{Cluster structures on $\kk[G^{u,w}]$ and $\kk[G]$}\label{subsec:Cluster-structures-on-G}

\subsubsection*{Root systems and generalized minors}

Following \cite[Section 2]{BerensteinFominZelevinsky05}\cite{Humphreys72},
we review basics of root systems and generalized minors for convenience
of the reader. 

Let $G$ be a simple, semisimple, complex algebra group of rank $|J|$.
Let $B$, $B_{+}$ be two opposite Borel subgroups. Let $N$ and $N_{-}$
be their unipotent radicals. Take the maximal torus $H=B\cap B_{-}$.

Let $\frg=\Lie(G)$ be the Lie algebra of $G$. $\frh=\Lie(H)$ is
the Cartan subalgebra. Recall that we have the exponential map $\exp:\frh\rightarrow G$.
Let $\Pi=\{\alpha_{a}|a\in J\}\subset\frh^{*}$ be the set of simple
roots. Denote the Cartan matrix by $C=(C_{ab})_{a,b\in J}$, where
$C_{ab}=\frac{2(\alpha_{a},\alpha_{b})}{(\alpha_{a},\alpha_{a})}$.
Let $\{\alpha_{a}^{\vee}|a\in J\}\subset\frh$ be the set of simple
coroots, such that $C_{ab}=\langle\alpha_{a}^{\vee},\alpha_{b}\rangle$.

The fundamental weights $\varpi_{a}\in\frh$ satisfy $\langle\alpha_{a}^{\vee},\varpi_{b}\rangle=\delta_{ab}$,
$\forall a,b\in J$. Equivalently, $s_{a}(\varpi_{b})=\varpi_{b}-\delta_{ab}\alpha_{a}$.

The Weyl group is $\Norm_{G}(H)/H$. We will identify it with the
corresponding Coxeter group $W$ generated by $s_{a}$, $a\in J$
(see Section \ref{subsec:Operations-on-signed}). The conjugation
action of the Weyl group on $H$ induces the action of $W$ on $\frh^{*}$,
such that the simple reflection is given by $s_{a}(\beta)=\beta-\langle\alpha_{i}^{\vee},\beta\rangle\alpha_{i}$,
$\forall\beta\in\frh^{*}$.

For any $a\in J$, let $\phi_{a}$ denote the $SL_{2}$-embedding
which sends upper (resp. lower) triangular matrices into $B$ (resp.
$B_{-}$). We identify $s_{a}$ with $\os_{a}H$ whose representative
is chosen as
\begin{align*}
\os_{a} & :=\phi_{a}\left(\begin{array}{cc}
0 & -1\\
1 & 0
\end{array}\right)\in G.
\end{align*}
For any $g=s_{a_{1}}\cdots s_{a_{r}}\in W$, let us denote $\og=\os_{a_{1}}\cdots\os_{a_{r}}$.
We can also identify $s_{a}$ with $H\dos_{a}$ where $\dos_{a}:=(\os_{a})^{-1}$,
and denote $\dog:=\dos_{a_{1}}\cdots\dos_{a_{r}}$. Note that $(\og)^{-1}=\doverline{g^{-1}}$. 

For any $g\in N_{-}HN$, denote the factorization $g=[g]_{-}[g]_{0}[g]_{+}$
such that $[g]_{-}\in N_{-}$, $[g]_{0}\in H$, $[g]_{+}\in N$.

Each weight $\gamma=\sum\gamma_{a}\varpi_{a}$, $\gamma_{a}\in\Z$,
determines a character $(\ )^{\gamma}$ on $H$ such that $(\exp(x))^{\gamma}:=e^{\langle x,\gamma\rangle}$,
$\forall x\in\frh$. $\forall u,w\in W$, we define the generalized
minor $\Delta_{u\varpi_{a},w\varpi_{a}}$ to be the regular function
on $G$ such that its value on $g\in\ou N_{-}HN\ow^{-1}$ is given
by: 
\begin{align*}
\Delta_{u\varpi_{a},w\varpi_{a}}(g) & =([\ou^{-1}g\ow]_{0})^{\varpi_{a}}.
\end{align*}
The function $\Delta_{u\varpi_{a},w\varpi_{a}}$ depends only on $(u\varpi_{a},w\varpi_{a})$,
and not on the choice of $(u,v)$. 

When we work at the quantum level ($\kk=\Q(v)$), we let $\kk[G]$
denote the quantized coordinate ring of $G$; see \cite{Kashiwara93}.
We also use $\Delta_{u\varpi_{a},w\varpi_{a}}$ to denote the corresponding
\emph{generalized quantum minor}; see \cite{BerensteinZelevinsky05}.

We will also consider the double Bruhat cell $G^{u',w'}=Bu'B\cap B_{-}w'B_{-}$.
Note that $G^{w_{0},w_{0}}$ is an open dense subset of $G$. We let
$\kk[G^{u',w'}]$ denote the (classical or quantized) coordinate ring
of $G^{u,w}$; see \cite{BerensteinFominZelevinsky05}\cite{BerensteinZelevinsky05}\cite{goodearl2016berenstein}.
By abuse of notation, we let $\Delta_{u\varpi_{a},w\varpi_{a}}$ denote
the restriction of the generalized (classical or quantum) minor on
$\kk[G^{u,w}]$.

\begin{eg}[Type $A_n$]

We briefly recall constructions for $G=SL_{n}(\C)$ following \cite{Humphreys72}.
\begin{itemize}
\item Choose the Euclidean vector space $\R^{n+1}$ with standard basis
$\{\epsilon_{a}|a\in[1,n+1]\}$ and the standard inner product $(\ ,\ )$. 
\item Choose the simple roots $\alpha_{a}=\epsilon_{a}-\epsilon_{a+1}$,
$\forall a\in[1,n]$. They span the subspace $\frh^{*}\subset\R^{l+1}$.
The Cartan matrix $C$ is given by $C_{ab}=\frac{2(\alpha_{a},\alpha_{b})}{(\alpha_{a},\alpha_{a})}$. 
\item The simple reflection $s_{a}$ is the reflection of $\R^{n+1}$ with
respect to the hyperplane whose normal vector is $\alpha_{a}$. In
the standard basis, it is represented by the following matrix 
\begin{align*}
\os_{a} & =\left(\begin{array}{cccccc}
\cdots\\
 & 1\\
 &  & 0 & -1\\
 &  & 1 & 0\\
 &  &  &  & 1\\
 &  &  &  &  & \cdots
\end{array}\right)
\end{align*}
 such that $\os_{a}(\epsilon_{a})=\epsilon_{a+1}$ and $\os_{a}(\epsilon_{a+1})=-\epsilon_{a}$.
\item Let $\{\epsilon_{a}^{\vee}|a\in[1,n+1]\}$ be the standard basis in
the dual space $(\R^{n+1})^{*}$. We have $\alpha_{a}^{\vee}=\epsilon_{a}^{\vee}-\epsilon_{a+1}^{\vee}$.
\item We can calculate $\varpi_{a}$ explicitly. For example, when $n=1$,
we have $\varpi_{1}=\frac{1}{2}\alpha_{1}=\frac{1}{2}(\epsilon_{1}-\epsilon_{2})$;
when $n=2$, we have $\varpi_{1}=\frac{1}{3}(2\alpha_{1}+\alpha_{2})=\frac{1}{3}(2\epsilon_{1}-\epsilon_{2}-\epsilon_{3})$,
$\varpi_{2}=\frac{1}{3}(\alpha_{1}+2\alpha_{2})=\frac{1}{3}(\epsilon_{1}+\epsilon_{2}-2\epsilon_{3})$.
\end{itemize}
Choose the maximal torus $H$ of $G=SL_{n+1}$ to be the subgroup
of diagonal matrices, and $N$ (reps. $N_{-}$) to be the subgroups
of unipotent upper (resp. lower) triangular matrices. Identify $\epsilon_{a}^{\vee}$
with the size $(n+1)$ elementary matrix $E_{aa}$, whose $(b,c)$-entries
are $(E_{aa})_{bc}=\delta_{ab}\delta_{ac}$, $\forall b,c\in[1,n+1]$.
Then we have $\exp(x)\in H$ for $x=\sum x_{i}\alpha_{i}^{\vee}$,
$x_{i}\in\C$. Each weight $\gamma=\sum\gamma_{i}\varpi_{i}$ determines
a character $(\ )^{\gamma}$ on $H$ such that $(\exp(x))^{\gamma}=e^{\langle x,\gamma\rangle}=e{}^{\sum x_{i}\gamma_{i}}$. 

Note that $W$ can be naturally identified with the permutation group
$S_{n+1}$ on $[1,n+1]$, such that $s_{i}$ is identified with the
transposition $(i,i+1)$. It is known that $\Delta_{u\varpi_{i},w\varpi_{i}}(g)$
equals the determinant of the submatrix $g_{u[1,i]\times w[1,i]}$,
See \cite[Section 2.2]{BerensteinFominZelevinsky05}. So we can denote
$\Delta_{u[1,i],w[1,i]}=\Delta_{u\varpi_{i},w\varpi_{i}}$. We will
further abbreviate $\Delta_{\{r_{1},\ldots,r_{s}\},\{c_{1},\ldots,c_{s}\}}$
by $\Delta_{r_{1}\ldots r_{s},c_{1}\ldots c_{s}}$.

\end{eg}

\subsubsection*{Cluster structures}

Consider the group $W\times W$. Denote the simple reflections in
the first copy by $s_{-a}$ and those in the second copy by $s_{a}$.
Let $\ubi=(\bi_{1},\ldots,\bi_{l})$ be a signed word such that it
is a reduced word for $(u,w)\in W\times W$, i.e., $s_{\bi_{1}}\cdots s_{\bi_{l}}$
is a minimal length factorization of $(u,w)$ into simple reflections.
Denote $(u_{\leq k},w_{\leq k})=s_{\bi_{1}}\cdots s_{\bi_{k}}$, $\forall k\in[1,l]$.
Let $(\bi_{-|J|+1},\ldots,\bi_{-1},\bi_{0})$ be any chosen Coxeter
word (i.e., any permutation of $J$). We have the seed $\dsd(\ubi)$
(Section \ref{sec:Seeds-associated-with-words}).

\begin{thm}[{\cite{qin2025partially}}]\label{thm:cluster-structure-k-G}

Take $\kk=\C$ or $\Q(v)$. Let $\ubi$ be any reduced word for $(w_{0},w_{0})$
and take $\sd=\dsd(\ubi)^{\op}$. Then we have an isomorphism $\kappa:\upClAlg(\sd)\simeq\kk[G^{u,w}]$
such that $\kappa x_{k}=\Delta_{\gamma_{k},\delta_{k}}$, $\forall k\in I$,
where
\begin{align}
\gamma_{k}= & \begin{cases}
\varpi_{\bi_{k}} & k\leq0\\
u_{\leq k}\varpi_{|\bi_{k}|} & k\in[1,l]
\end{cases} & \text{and\ensuremath{\qquad}} & \delta_{k}= & \begin{cases}
w^{-1}\varpi_{\bi_{k}} & k\leq0\\
w^{-1}w_{\leq k}\varpi_{|\bi_{k}|} & k\in[1,l]
\end{cases}.\label{eq:kappa-minor}
\end{align}
And the compatible Poisson structure $\lambda(\sd)$ has the skew-symmetric
quantization matrix $\Lambda$ given by 
\begin{align*}
\Lambda_{kj} & =(\gamma_{k},\gamma_{j})-(\delta_{k},\delta_{j}),\ \forall k>j.
\end{align*}

Moreover, take $\sd'=\dsd(\ubj)^{\op}$ for another reduced signed
word of $(w_{0},w_{0}),$ and let $\kappa'$ denote the corresponding
isomorphism $\kappa':\upClAlg(\sd')\simeq\kk[G^{u,w}]$, then $(\kappa')^{-1}\kappa(x_{k})$
are the cluster variables of the seed $\dsd(\ubi)^{\op}$ in $\bUpClAlg(\sd')$.

Finally, in the special case $(u,w)=(w_{0},w_{0})$, $\kappa$ restrict
a cluster structure on $\kk[G]$:
\begin{align*}
\kappa:\bUpClAlg(\sd)\simeq\kk[G].
\end{align*}

\end{thm}

By Theorem \ref{thm:cluster-structure-k-G}, different choices of
the the reduced words provide the same cluster structures on $\kk[G^{u,w}]$,
and the same is true for $\kk[G]$. We will omit the symbol $\kappa$
and identify $\kk[G^{u,w}]$ and $\kk[G]$ with the corresponding
cluster algebras, respectively.

\begin{rem}

The statements about $\C[G^{u,w}]$ was first proved in \cite{BerensteinFominZelevinsky05}\cite{shen2021cluster}.
The statements for quantum $\kk[G^{u,w}]$ was partially proved in
\cite{BerensteinZelevinsky05}\cite{goodearl2016berenstein}\cite{GeissLeclercSchroeer11}.
The cluster structure on $\C[SL_{n}]$ could be found in \cite{fomin2020introduction}. 

Based on an approach different from that of \cite{qin2025partially},
\cite{oya2025note} proved that $\C[G]=\bClAlg(\sd)=\bUpClAlg(\sd)$
when $G$ is not of type $F_{4}$.

\end{rem}

\begin{eg}[{\cite[Example 7.8]{qin2023analogs}}]

Let us consider the case $G=SL_{2}$ and choose the reduced word $\ubi=(1,-1)$
for $(w_{0},w_{0})\in W\times W$. The quiver for the seed $\dsd(\ubi)$
is given in Figure \ref{fig:quiver-SL2}. Take the (classical or quantum)
seed $\sd=\dsd(\ubi)^{\op}$. Then we have an isomorphism $\kappa:\bUpClAlg(\sd)\simeq\kk[SL_{3}]$,
such that $\kappa$ send $x_{k}$, $k\in[0,2]$, to:
\begin{align*}
\Delta_{1,2},\Delta_{1,1},\Delta_{2,1}.
\end{align*}
The frozen variables are $\Delta_{1,2},\Delta_{2,1}$. The quantization
matrix $\Lambda$ is
\begin{align*}
\Lambda & =\left(\begin{array}{ccc}
0 & -1 & 0\\
1 & 0 & 1\\
0 & -1 & 0
\end{array}\right).
\end{align*}
The matrix $(b_{ij})_{i,j\in[-1,1]}$ is
\begin{align*}
\tB= & \left(\begin{array}{ccc}
0 & -1 & 0\\
1 & 0 & 1\\
0 & -1 & 0
\end{array}\right).
\end{align*}

Let us denote $u=x_{0}(\sd)$, $x=x_{1}(\sd)$, $v=x_{2}(\sd)$, $y=x_{1}(\mu_{1}\sd)$.
The common triangular basis of $\bClAlg=\bUpClAlg$ is $\{[u^{n}x^{m}v^{l}]|n,m,l\in\N\}\sqcup\{[u^{n}y^{m}v^{l}]|n,m,l\in\N,m>0\}$.
It coincides with the global crystal basis of the $\Q(v)[SL_{2}]$
in \cite[Section 9]{Kashiwara93} up to $v$-multiples.

\end{eg}

\begin{eg}

Let us consider the case $G=SL_{3}$ and choose the reduced word $\ubi=(1,-1,2,-2,1,-1)$
for $(w_{0},w_{0})\in W\times W$. The quiver for the seed $\dsd(\ubi)$
is given in Figure \ref{fig:SL3}. Take the (classical or quantum)
seed $\dsd=\dsd(\ubi)^{\op}$. We have $\kappa:\bUpClAlg(\sd)\simeq\kk[SL_{3}]$,
such that $\kappa$ send $x_{k}$, $k\in[-1,6]$, to:
\begin{align*}
\Delta_{1,3},\Delta_{12,23},\Delta_{1,2},\Delta_{2,2},\Delta_{12,12},\Delta_{23,12},\Delta_{2,1},\Delta_{3,1}.
\end{align*}
The frozen variables are $\Delta_{1,3},\Delta_{12,23},\Delta_{23,12},\Delta_{3,1}$.

\end{eg}

\begin{eg}[{\cite[Example 2.9]{BerensteinFominZelevinsky05}}]

Let us consider the case $G=SL_{3}$ and choose the reduced word $\ubi=(1,2,1,-1,-2,-1)$
for $(w_{0},w_{0})\in W\times W$. The quiver for the seed $\dsd(\ubi)$
is given in Figure \ref{fig:SL3-BFZ}. Take the (classical or quantum)
seed $\sd=\dsd(\ubi)^{\op}$. We have $\kappa:\bUpClAlg(\sd)\simeq\kk[SL_{3}]$,
such that $\kappa$ send $x_{k}$, $k\in[-1,6]$, to:
\begin{align*}
\Delta_{1,3},\Delta_{12,23},\Delta_{1,2},\Delta_{12,12},\Delta_{1,1},\Delta_{2,1},\Delta_{23,12},\Delta_{3,1}.
\end{align*}
The frozen variables are $\Delta_{1,3},\Delta_{12,23},\Delta_{23,12},\Delta_{3,1}$.

\begin{figure}
\caption{The quiver for an infinite seed of type $A_{1}$ in \cite{geiss2024representations}}
\label{fig:SL3-BFZ}

\begin{tikzpicture}
 [scale=1.5,node distance=48pt,on grid,>={Stealth[length=4pt,round]},bend angle=45, inner sep=0pt]

\node[frozen] (q0) at (113.8pt,42.675pt) {0};

\node[frozen] (q-1) at (99.575pt,14.225pt) {-1};
\node[unfrozen] (q1) at (128.025pt,14.225pt) {1};
\node[unfrozen] (q2) at (142.25pt,42.675pt) {2};
\node[unfrozen] (q3) at (5.5,14.225pt) {3};
\node[unfrozen] (q4) at (170.7pt,14.225pt) {4};
\node[frozen] (q5) at (184.925pt,42.675pt) {5};
\node[frozen] (q6) at (199.15pt,14.225pt) {6};

\draw[->,dashed]  (q0) edge (q-1);
\draw[->,dashed]  (q5) edge (q6);

\draw[->]  (q-1) edge (q1);
\draw[->]   (q1) edge (q0);
\draw[->]   (q0) edge (q2);
\draw[->]   (q2) edge (q1);
\draw [->]  (q1) edge (q3);
\draw [->]  (q4) edge (q3);
\draw [->]  (q3) edge (q2);

\draw [->]  (q2) edge (q4);
\draw [->]  (q5) edge (q2);

\draw [->]  (q4) edge (q5);
\draw [->]  (q6) edge (q4);
\end{tikzpicture}
\end{figure}

\end{eg}

\begin{eg}

Let us consider the case $G=SL_{3}$ and choose the reduced word $\ubi=(1,2,1)$
for $(e,w_{0})\in W\times W$. The quiver for the seed $\dsd(\ubi)$
is given in Figure \ref{fig:SL3-dBS}. Take the (classical or quantum)
seed $\sd=\dsd(\ubi)^{\op}$. We have $\kappa:\upClAlg(\sd)\simeq\kk[SL_{3}^{e,w_{0}}]$,
such that $\kappa$ send $x_{k}$, $k\in[-1,3]$, to:
\begin{align*}
\Delta_{1,3},\Delta_{12,23},\Delta_{1,2},\Delta_{12,12},\Delta_{1,1}.
\end{align*}
We have $I_{\fv}=\{-1,0,2,3\}$, $I_{\ufv}=\{1\}$, $\tB=\left(\begin{array}{c}
-1\\
1\\
0\\
-1\\
1
\end{array}\right)$. 

\begin{figure}
\caption{The quiver for an infinite seed of type $A_{1}$ in \cite{geiss2024representations}}
\label{fig:SL3-dBS}

\begin{tikzpicture}
 [scale=1.5,node distance=48pt,on grid,>={Stealth[length=4pt,round]},bend angle=45, inner sep=0pt]

\node[frozen] (q0) at (113.8pt,42.675pt) {0};

\node[frozen] (q-1) at (99.575pt,14.225pt) {-1};
\node[unfrozen] (q1) at (128.025pt,14.225pt) {1};
\node[frozen] (q2) at (142.25pt,42.675pt) {2};
\node[frozen] (q3) at (5.5,14.225pt) {3};

\draw[->,dashed]  (q0) edge (q-1);

\draw[->]  (q-1) edge (q1);
\draw[->]   (q1) edge (q0);
\draw[->]   (q0) edge (q2);
\draw[->]   (q2) edge (q1);
\draw [->]  (q1) edge (q3);
\draw [->,dashed]  (q3) edge (q2);

\end{tikzpicture}
\end{figure}

Note that the cluster variables are 
\begin{align*}
\Delta_{1,3} & =\Delta_{\varpi_{1},w_{0}^{-1}\varpi_{1}}=\Delta_{\varpi_{1},\varpi_{1}-\alpha_{1}-\alpha_{2}}=\Delta_{\frac{2}{3}\alpha_{1}+\frac{1}{3}\alpha_{2},-\frac{1}{3}\alpha_{1}-\frac{2}{3}\alpha_{2}}\\
\Delta_{12,23} & =\Delta_{\varpi_{2},w_{0}^{-1}\varpi_{2}}=\Delta_{\varpi_{2},\varpi_{2}-\alpha_{1}-\alpha_{2}}=\Delta_{\frac{1}{3}\alpha_{1}+\frac{2}{3}\alpha_{2},-\frac{2}{3}\alpha_{1}-\frac{1}{3}\alpha_{2}}\\
\Delta_{1,2} & =\Delta_{\varpi_{1},s_{1}s_{2}\varpi_{1}}=\Delta_{\varpi_{1},\varpi_{1}-\alpha_{1}}=\Delta_{\frac{2}{3}\alpha_{1}+\frac{1}{3}\alpha_{2},-\frac{1}{3}\alpha_{1}+\frac{1}{3}\alpha_{2}}\\
\Delta_{12,12} & =\Delta_{\varpi_{2},s_{1}\varpi_{2}}=\Delta_{\frac{1}{3}\alpha_{1}+\frac{2}{3}\alpha_{2},\frac{1}{3}\alpha_{1}+\frac{2}{3}\alpha_{2}}\\
\Delta_{1,1} & =\Delta_{\varpi_{1},\varpi_{1}}=\Delta_{\frac{2}{3}\alpha_{1}+\frac{1}{3}\alpha_{2},\frac{2}{3}\alpha_{1}+\frac{1}{3}\alpha_{2}}.
\end{align*}
The quantization matrix is 
\begin{align*}
\Lambda & =\left(\begin{array}{ccccc}
0 & 0 & -1 & -1 & -1\\
0 & 0 & 0 & -1 & -1\\
1 & 0 & 0 & 0 & -1\\
1 & 1 & 0 & 0 & 0\\
1 & 1 & 1 & 0 & 0
\end{array}\right).
\end{align*}

\end{eg}

\subsection{Cluster algebras from representations of simply-laced quantum affine
algebras}\label{subsec:Cluster-algebras-from-q-aff-alg}

Let $\frg$ denote a finite dimensional complex semisimple Lie algebra
and whose Cartan matrix $C$ is symmetric ($\frg$ is called simply
laced). Let $U_{\epsilon}(\hfrg)$ denote the quantum affine algebra,
which is a Hopf algebra, where the parameter $\epsilon\in\C^{\times}\backslash\{0\}$
is not a root of unity. Let $\cO$ denote the abelian category consisting
of its (type $1$) finite dimensional representations, which has the
tensor product $\otimes$. The simple modules $S$ in $\cO$ are parameterized
by their Drinfeld polynomials $P_{S}=\{P_{a}(z)|a\in J\}$, where
$z$ is a formal variable, such that $P_{a}(z)$ take the form $P_{a}(z)=\prod_{i}(1-\lambda_{i}z)$
for some $\lambda_{i}\in\C^{\times}$.

We can always choose a height function $\xi:J\rightarrow\Z$ such
that $|\xi(a)-\xi(b)|=1$ whenever $C_{ab}\neq0$. Let $\cO_{\Z}$
be the abelian subcategory of $\cO$ such that the roots of the Drinfeld
polynomials $P_{a}$ of the simple objects in $\cO_{\Z}$ belong to
$\epsilon^{2\Z+\xi(a)}$. $\cO_{\Z}$ is a tensor subcategory.

Denote $\hJ=\{(a,d)|a\in J,d\in\xi(a)+2\Z\}$ and $W=\oplus_{(a,d)\in\hJ}\N e_{a,d}$.
For $w=\sum w_{a,d}e_{a,d}\in W$, let $S(w)$ denote the simple module
whose Drinfeld polynomials are $P_{a}(z)=\prod_{d\in\xi(a)+2\Z}(1-\varepsilon^{d}z)^{w_{a,d}}$,
$a\in J$. In particular, the modules $S(e_{a,d})$ are called the
fundamental modules, denoted $W_{a,d}$. For $k\in\Z_{>0}$, the modules
$S(e_{a,d}+e_{a,d+2}+\cdots+e_{a,d+2k-2})$ are called the Kirillov-Reshetikhin
modules, denoted $W_{a,d}^{k}$. We also introduce the standard module
arising from quiver varieties
\begin{align}
M(w) & :=\overrightarrow{\otimes}_{d}(\otimes_{a}W_{a,d}{}^{\otimes w_{a,d}}),\label{eq:std-mod}
\end{align}
where we take the increasing order in $\overrightarrow{\otimes}_{d}$
(different orders in $\otimes_{a}$ produce the same result); see
\cite[Axiom 3]{Nakajima04}\cite[Corollary 7.16]{varagnolo2002standard}.

Although $V\otimes W\ncong W\otimes V$ in $\cO_{\Z}$ in general,
the Grothendieck ring $K_{0}(\cO_{\Z})$ is a commutative algebra.
For simplicity, we will denote the isoclass of any module $V$ in
$K_{0}(\cO_{\Z})$ by $V$. Recall that there is an embedding called
the $q$-character $\chi_{q}:K_{0}(\cO_{\Z})\hookrightarrow\cY:=\Z[Y_{a,d}^{\pm}]_{(a,d)\in\hJ}$.
Define the Laurent monomial 
\begin{align*}
A_{a,h} & :=Y_{a,h+1}Y_{a,h-1}\prod_{b\neq a}Y_{b,h}^{C_{ba}},
\end{align*}
where $(b,h-1)\in\hJ$. Denote $Y^{w}=\prod Y_{a,d}^{w_{a,d}}$ and
$A^{-v}=\prod A_{b,h}^{-v_{b,h}}$ for $v_{b,h}\in\Z$. We introduce
the dominance order $\prec$ on $W$ such that $w'\preceq w$ if $Y^{w'}=Y^{w}(A^{-1})^{v}$
for some $v=(v_{a,h})\geq0$. Then $\chi_{q}(S(w))$ is $w$-pointed,
i.e., $Y^{w}$ is its unique $\prec$-leading term with coefficient
$1$. Denote $\chi_{q}(V)$ by $[V]$. We will identify $K_{0}(\cO_{\Z})$
with its image under $\chi_{q}$.

We can introduce a quantum torus algebra $\cY_{t}:=\Z[t^{\pm}][Y_{a,d}]_{(a,d)\in\hJ}$
using the $t$-twisted multiplication $*$ in \cite[Section 7.3]{qin2017triangular}\cite[(6)]{HernandezLeclerc11},
viewed as a $t$-deformation of $\cY$. The bar-involution $\overline{(\ )}$
is the $\Z$-linear automorphism on $\cY_{t}$ such that $\overline{t^{\alpha}Y^{w}}=t^{-\alpha}Y^{w}$.
Following \cite{Nakajima04}\cite{VaragnoloVasserot03}\cite{Hernandez02},
we have a $Z$-linear injective map, called a $t$-analog of $q$-characters:
\begin{align*}
\chi_{q,t} & :K_{0}(\cO_{\Z})\hookrightarrow\cY_{t}.
\end{align*}
At the classical limit $t\mapsto1$, $\chi_{q,t}$ becomes $\chi_{q}$.
We further have 
\begin{align}
\chi_{q,t}M(w) & =t^{\alpha}\overleftarrow{\prod_{d}}(\prod_{a}\chi_{q,t}(W_{a,d})^{w_{a,d}}).\label{eq:char-std-mod}
\end{align}
for appropriate $\alpha\in\Z$ such $M(w)$ is $w$-pointed, where
the product $\overleftarrow{\prod_{d}}$ is taking in the \textbf{decreasing}
order.

\begin{rem}

The equation \eqref{eq:char-std-mod} was known in \cite{Nakajima04}\cite{varagnolo2002standard}
for a slightly different $t$-deformation and the increasing order
$\overrightarrow{\prod_{d}}$. It holds for our $t$-deformation and
the decreasing order $\overleftarrow{\prod_{d}}$ ; see \cite[Sections 5.6]{HernandezLeclerc11}\cite[Proposition 5.3.6]{Qin12}.

\end{rem}

Let $K_{t}(\cO_{\Z})^{\op}$ denote the free $\Z[t^{\pm}]$-module
spanned by $\{\chi_{q,t}(S)|S\text{ is simple in }\cO_{\Z}\}$. It
is a subalgebra of $\cY_{t}$. We will call its opposite algebra $K_{t}(\cO_{\Z})$
the quantum Grothendieck ring of $\cO_{\Z}$, which is viewed as a
$t$-deformation of $K_{0}(\cO_{\Z})$. Correspondingly, $K_{t}(\cO_{\Z}^{\op})=K_{t}(\cO_{\Z})^{\op}$
is called the quantum Grothendieck ring of $\cO_{\Z}^{\op}$. Recall
that we have the canonical anti-isomorphism $\iota:\cY_{t}\simeq\cY_{t}^{\op}$
between $\Z[v^{\pm}]$-algebras, such that $\iota(Y^{w}):=Y^{w}$
(see \eqref{eq:iota-opposite}). It induces $\iota:K_{t}(\cO_{\Z})^{\op}\simeq K_{t}(\cO_{\Z})$.
For any $V\in\cO_{\Z}$, we denote $[V]_{t}:=\iota\chi_{q,t}(V)\in K_{t}(\cO_{\Z})$.

\begin{rem}

Our quantum Grothendieck ring $K_{t}(\cO_{\Z})$ is the opposite algebra
of the one used in \cite{HernandezLeclerc11}\cite{qin2017triangular}.
We make this choice because it leads to the same product order on
the both sides below in $K_{t}(\cO_{\Z})$:
\begin{align*}
[\overrightarrow{\otimes}_{d}(\otimes_{a}W_{a,d}{}^{\otimes w_{a,d}})]_{t} & =t^{\alpha}\overrightarrow{\prod_{d}}(\prod_{a}[W_{a,d}]_{t}^{w_{a,d}}).
\end{align*}
Correspondingly, if an algebra $\alg$ was previously categorified
by $\cO_{\Z}$ in the convention of \cite{HernandezLeclerc11}\cite{qin2017triangular},
$\alg^{\op}$ is said to be categorified by $\cO_{\Z}$ in our convention.

\end{rem}

\begin{thm}[\cite{Nakajima04}]\label{thm:qt-char-simples}

For any $w\in W$, the following statements are true in the quantum
Grothendieck ring $K_{t}(\cO_{\Z})$.
\begin{enumerate}
\item $\overline{[S(w)]_{t}}=[S(w)]_{t}$.
\item $[M(w)]_{t}=[S(w)]_{t}+\sum_{w'<w}a_{w,w'}(t)[S(w')]_{t}$.
\item $a_{w,w'}(t)\in t^{-1}\N[t^{-1}]$.
\item $a_{w,w'}(1)$ equals the multiplicity $[M(w):S(w')]$.
\end{enumerate}
\end{thm}

\subsubsection*{Cluster structure}

Choose any height function $\xi$. Let $\uc=(c_{1},\ldots,c_{|J|})$
be a Coxeter word and $\ueta=(\eta_{i})_{i\in[r,s]}$ be a $\uc$-adapted
word, i.e., $\ueta$ is a subsequence of $\uc^{N}$ for some $N$. 

For any interval $[j,k]\subset[r,s]$ such that $\eta_{j}=\eta_{k}=:a$,
let $W_{[j,k]}(\xi)$ denote the Kirillov-Reshetikhin module $W_{a,-2k+\xi(a)}^{k-j+1}=S(e_{a,-2k+\xi(a)}+e_{a,-2k+2+\xi(a)}+\cdots+e_{a,-2j+\xi(a)})$.
Denote the fundamental module by $W_{k}(\xi):=W_{[k,k]}(\xi)=S(e_{a,-2k+\xi(a)})$.
Let $\cO_{\ueta}=\cO_{\ueta}(\xi)$ be the smallest tensor subcategory
of $\cO_{\Z}$ containing the fundamental modules $W_{k}(\xi)$, $\forall k\in[r,s]$,
and $K_{t}(\cO_{\ueta})$ the corresponding subalgebra of $K_{t}(\cO_{\Z})$
generated by $[W_{k}(\xi)]_{t}$, $k\in[r,s]$.

\begin{rem}

Hernandez and Leclerc called $\cO_{(\uc)^{l+1}}$ a level-$l$ subcategory,
denoted $\cO_{l}$ (\cite{HernandezLeclerc09}). When $\overrightarrow{w}$
is a reduced word of $w_{0}$ and it is $\uc$-adapted, they also
considered the subcategory $\cO_{\overrightarrow{w}}$ and showed
that $K_{t}(\cO_{\overrightarrow{w}})^{\op}\otimes\Q(v)$ is isomorphic
to the quantized coordinate ring $\Q(v)[N]$ \cite{HernandezLeclerc11}.

\end{rem}

When $S(e_{a,d})$ is the fundamental module $W_{k}(\xi)$ for $k\in[r,s]$,
denote $\beta_{k}:=e_{a,d}$. As in \eqref{eq:wt-to-g-vector}, denote
the linear map $\theta^{-1}:\oplus_{k\in[r,s]}\Z\beta_{k}\simeq\Mcirc$
such that 
\begin{align*}
\theta^{-1}(\beta_{k}) & =f_{k}-f_{k[-1]},\forall k\in[r,s],
\end{align*}
where $f_{k[-1]}=0$ if $k=k^{\min}$ in $[r,s]$.

\begin{thm}[\cite{qin2017triangular}]\label{thm:cluster-str-HL-ADE}

Take $\sd=\rsd(\ueta)^{\op}$ and the canonical anti-isomorphism $\iota:\bUpClAlg(\sd^{\op})\simeq\bUpClAlg(\sd)$.
Then we have the $\Z$-algebra isomorphism $\kappa:\bUpClAlg(\sd)\simeq K_{t}(\cO_{\ueta})$
such that $\kappa(v)=t$, $\kappa\iota(W_{[j,k]})=[W_{[j,k]}(\xi)]_{t}$.
Moreover, the common triangular basis $\can$ of $\bUpClAlg(\sd)$
is $\{(\kappa)^{-1}[S(w)]_{t}|S(w)\text{ is simple in }\cO_{\ueta}\}$,
such that $(\kappa)^{-1}[S(w)]_{t}$ is $\theta^{-1}(w)$-pointed.

\end{thm}

\begin{rem}

Note that we always have $\bUpClAlg(\sd)=\bClAlg(\sd)$.

In the classical case ($\kk=\Z$), \cite{HernandezLeclerc09} showed
that $\cO_{(\uc)^{l+1}}$ categorifies $\bClAlg(\sd)$ for $\frg$
of type $A_{1}$ or $l=1$. \cite{qin2017triangular} proved this
categorification result (at the classical or quantum level) for all
simply laced $\frg$. The categorification result is true for all
semisimple $\frg$ by \cite{kashiwara2021monoidal}.

Extending the range $[r,s]$ for an appropriate signed word $\ubi=(\bi_{k})_{k\in[r,s]}$
to $(-\infty,+\infty)$ as in Section \ref{sec:Extension-to-infinite},
one could easily extend the statements in Theorem \ref{thm:tri-basis-is-canonical}(2)
to $K_{t}(\cO_{\Z})$; see \cite{qin2024infinite}. One could still
establish the cluster structure when $\frg$ is not simply laced (see
\cite{hernandez2013cluster}). In this case, $\kappa\can$ still equals
the Kazhdan-Lusztig-type basis $\{[S(w)]_{t}|S(w)\text{ is simple}\}$
(\cite{qin2023analogs}). But it is still conjectured that $\chi_{q,t}(S(w))|_{t\mapsto1}=\chi_{q}(S(w))$
(\cite{fujita2022isomorphisms} proved it in type $B$).

\end{rem}

\begin{eg}

Let us continue the seed $\sd=\rsd(\ubi)$ in Example \ref{eg:HL-A2-inf},
where $\uc=(1,2,3)$, $\ubi=(\bi_{k})_{k\in[1,\infty)}=(\uc)^{\infty}$.
Choose $\xi(a)=-a$ for $a\in\{1,2,3\}$, such that $\xi(c_{j})\geq\xi(c_{k})$
when $j<k$. We have $\kappa\iota:\bUpClAlg(\sd)\subset K_{t}(\cO_{\Z})$.
The interval variables are $W_{[a+2r,a+2s]}$ for $0\leq r\leq s$.
Then we have 
\begin{align*}
\kappa\iota W_{[a+2r,a+2s]} & =[S(e_{a,-2s+\xi(a)}+e_{a,-2s+2+\xi(a)}+\cdots+e_{a,-2r+\xi(a)}]_{t}\in K_{t}(\cO_{\Z}).
\end{align*}

\end{eg}
\input{main.bbl}


\end{document}

%% file: macroAdd.tex
\usepackage{amssymb}

\usepackage{tikz-cd}

\usepackage[normalem]{ulem}
\usepackage{amsfonts}\usepackage{amscd}\usepackage{tikz-cd}
\usepackage{verbatim}
\usepackage{listings}
\usepackage{epsfig}
\usepackage{url}
\usepackage{epic}
\usepackage{latexsym}
\usepackage{graphicx}
\usepackage{xcolor}

\usepackage{lcd}
\usepackage{mathrsfs}

\usepackage{hyperref}
\hypersetup{%
pdftitle={Preprint},
pdfauthor={},
pdfkeywords={},
bookmarksnumbered,
pdfstartview={FitH},
breaklinks=true,
colorlinks=true,
urlcolor=blue,
citecolor=blue,
linktocpage=true
}%
\usepackage[hypcap=true]{caption}

\usepackage{breakurl}

\usepackage{lscape}
\usepackage{tabularx}
\usepackage{marginnote}
\usepackage[all]{xy}
\usepackage{enumitem}
\usepackage{tipa}
\usepackage{upgreek}

\usepackage{tikz}
\usetikzlibrary{positioning,shapes,shadows,arrows,snakes,arrows.meta}

\pgfdeclarelayer{edgelayer}
\pgfdeclarelayer{nodelayer}
\pgfsetlayers{edgelayer,nodelayer,main}

\tikzstyle{none}=[inner sep=0pt]
\tikzstyle{black box}=[draw=black, fill=black!25]
\tikzstyle{white box}=[draw=black, fill=white]
\tikzstyle{black circle}=[circle,draw=black!50, fill=black!25]
\tikzstyle{red circle}=[circle,draw=red!50, fill=red!25]
\tikzstyle{blue circle}=[circle,draw=blue!50, fill=blue!25]
\tikzstyle{green circle}=[circle,draw=green!50, fill=green!25]
\tikzstyle{yellow circle}=[circle,draw=yellow!50, fill=yellow!25]

\tikzstyle{unfrozen}=[circle,inner sep=1pt,minimum size=1pt,draw=black!100,fill=red!50]
\tikzstyle{frozen}=[rectangle,inner sep=1pt,minimum size=12pt,draw=black!75,fill=cyan!50]
\tikzstyle{boundary}=[-,draw=cyan]
\tikzstyle{internal}=[-,draw=red]
\tikzstyle{pre}=[<-,shorten <=1pt,>={Stealth[round]},semithick]
\tikzstyle{post}=[->,shorten >=1pt,>={Stealth[round]},semithick]
\tikzstyle{point}= [circle,inner sep=0pt, outer sep=1.5pt,minimum size=1.5pt,draw=black!100,fill=black!100]
\tikzstyle{every label}= [black]

\DeclareMathOperator{\Spec}{Spec}

\DeclareMathOperator{\Id}{Id}

\DeclareMathOperator{\prin}{prin}

\DeclareMathOperator{\Lie}{Lie}

\def\!{\mskip-\thinmuskip}
\let\?=\overline
\let\:=\colon
\let\llb=\llbracket
\let\rrb=\rrbracket

\newcommand{\kk}{\Bbbk}

\newcommand{\mm}{{\bf m}}

\newcommand{\sd}{{\bf s}}

\newcommand{\red}[1]{{\color{red} #1}}

\newcommand{\F}{{\mathbb F}}

\newcommand{\Q}{{\mathbb Q}}
\newcommand{\R}{{\mathbb R}}
\newcommand{\Z}{{\mathbb Z}}
\newcommand{\N}{{\mathbb N}}

\newcommand{\pr}{\operatorname{pr}} 

\newcommand{\cC}{\mathcal{C}}
\newcommand{\cO}{\mathcal{O}}
\newcommand{\cY}{\mathcal{Y}}

\newcommand{\op}{{\mathrm{op}}}
\newcommand{\codeg}{\mathrm{codeg}}
\newcommand{\suppDim}{\mathrm{suppDim}}

\newcommand{\Hf}{\frac{1}{2}}

\newcommand{\ufv}{{\operatorname{uf}}}
\newcommand{\fv}{{\mathrm f}}

\newcommand{\uk}{\underline{k}}

\newcommand{\var}{\operatorname{var}}

\newcommand{\tQ}{\widetilde{Q}}

\newcommand{\alg}{\mathsf{A}}
\newcommand{\clAlg}{\boldsymbol{A}}
\newcommand{\upClAlg}{\boldsymbol{U}}
\newcommand{\bClAlg}{\overline{\clAlg}}
\newcommand{\bUpClAlg}{\overline{\upClAlg}}

\newcommand{\can}{\textbf{L}}
\newcommand{\bCan}{\textbf{B}}
\newcommand{\dCan}{\textbf{B}^*}

\newcommand{\Inj}{\textbf{I}}

\newcommand{\AVar}{\mathscr{A}}

\newcommand{\frz}{\mathfrak{f}}

\newcommand{\Br}{\mathrm{Br}}

\renewcommand{\Id}{\mathrm{Id}}

\newcommand{\ueta}{\underline{\eta}}

\newcommand{\bi}{\mathbf{i}}
\newcommand{\bj}{\mathbf{j}}

\newcommand{\ubi}{\underline{\bi}}

\newcommand{\ubj}{\underline{\bj}}

\newcommand{\uc}{\underline{c}}

\newcommand{\rsd}{\dot{\sd}}
\newcommand{\dsd}{\ddot{\sd}}
\newcommand{\seq}{\boldsymbol{\mu}}
\newcommand{\diag}{\delta}

\newcommand{\sym}{\mathsf{d}}

\newcommand{\hJ}{\hat{J}}

\newcommand{\LP}{\boldsymbol{LP}}
\newcommand{\bLP}{\overline{\LP}}
\newcommand{\hLP}{\widehat{\LP}}

\newcommand{\hV}{\widehat{V}}

\newcommand{\base}{\textbf{S}}
\newcommand{\std}{\textbf{M}}
\newcommand{\stdBasis}{\textbf{M}}

\newcommand{\cM}{\mathcal{M}}
\newcommand{\cT}{\mathcal{T}}
\newcommand{\cZ}{\mathcal{Z}}

\newcommand{\tC}{\widetilde{C}}
\newcommand{\tN}{\widetilde{N}}
\newcommand{\tw}{\widetilde{w}}

\newcommand{\frg}{\mathfrak{g}}
\newcommand{\frn}{\mathfrak{g}}
\newcommand{\frh}{\mathfrak{h}}

\newcommand{\hfrg}{\hat{\mathfrak{g}}}

\newcommand{\Norm}{\mathrm{Norm}}

\newcommand{\simp}{\mathsf{S}}
\newcommand{\bq}{\mathbf{q}}
\newcommand{\BZ}{\mathrm{BZ}}

\newcommand{\tropM}{M^\mathrm{trop}}
\newcommand{\Mtrop}{M^\mathrm{trop}}
\newcommand{\domM}{\overline{M}}
\newcommand{\domMtrop}{\domM^\mathrm{trop}}

\newcommand{\Mcirc}{{M^{\circ}}}

\newcommand{\tB}{\widetilde{B}}

\newcommand{\cF}{\mathcal{F}}
\newcommand{\cR}{\mathcal{R}}

\usetikzlibrary{calc}
\newcommand\doverline[1]{%
\tikz[baseline=(nodeAnchor.base)]{
    \node[inner sep=0] (nodeAnchor) {$#1$}; 
    \draw[line width=0.1ex,line cap=round] 
        ($(nodeAnchor.north west)+(0.0em,0.2ex)$) 
            --
        ($(nodeAnchor.north east)+(0.0em,0.2ex)$) 
        ($(nodeAnchor.north west)+(0.0em,0.5ex)$) 
            --
        ($(nodeAnchor.north east)+(0.0em,0.5ex)$) 
    ;
}}
\newcommand{\os}{\overline{s}}
\newcommand{\ou}{\overline{u}}
\newcommand{\ow}{\overline{w}}
\newcommand{\og}{\overline{g}}
\newcommand{\dos}{\doverline{s}}

\newcommand{\dog}{\doverline{g}}

\newcommand{\cRing}{\mathsf{R}}

\newcommand{\col}{\mathrm{col}}

\newcommand{\spec}{\Spec}

\newcommand{\lex}{\mathrm{lex}}
\newcommand{\rev}{\mathrm{rev}}

%% file: macroHybrid.tex
\ifdefined\supp
\else
\newcommand{\supp}{\operatorname{supp}}
\fi

\ifdefined\G
    \renewcommand{\G}{{\mathbb G}}
\else
   \newcommand{\G}{{\mathbb G}}
\fi

\ifdefined\C
    \renewcommand{\C}{{\mathbb C}}
\else
   \newcommand{\C}{{\mathbb C}}
\fi

\ifdefined\thm
\else
	\theoremstyle{plain}
 	\newtheorem{thm}{Theorem}[section]
\fi

\ifdefined\prop
\else
	\theoremstyle{plain}
 	\newtheorem{prop}[thm]{Proposition}
\fi\ifdefined\cor
\else
	\theoremstyle{plain}
 	\newtheorem{cor}[thm]{Corollary}
\fi

\ifdefined\lem
\else
	\theoremstyle{plain}
 	\newtheorem{lem}[thm]{Lemma}
\fi

\ifdefined\def
\else
	\theoremstyle{definition}
	\newtheorem{def}[thm]{Definition}
\fi

\ifdefined\defn
\else
	\theoremstyle{definition}
	\newtheorem{defn}[thm]{Definition}
\fi

\ifdefined\example
\else
	\theoremstyle{definition}
    \newtheorem{example}[thm]{Example}
\fi

\ifdefined\eg
\else
		\theoremstyle{definition}
        \newtheorem{eg}[thm]{Example}
\fi

\ifdefined\rem
\else
		\theoremstyle{remark}
        \newtheorem{rem}[thm]{Remark}
\fi

\newtheorem{asm}[thm]{Assumption}

%% file: main.bbl
\newcommand{\etalchar}[1]{$^{#1}$}
\def\cprime{$'$}
\providecommand{\bysame}{\leavevmode\hbox to3em{\hrulefill}\thinspace}
\providecommand{\MR}{\relax\ifhmode\unskip\space\fi MR }
\providecommand{\MRhref}[2]{%
  \href{http://www.ams.org/mathscinet-getitem?mr=#1}{#2}
}
\providecommand{\href}[2]{#2}